\documentclass[12pt,reqno]{amsart}
\usepackage{mathrsfs}
\usepackage{amsgen}
\usepackage{amscd}
\usepackage{amsmath}
\usepackage{latexsym}
\usepackage{amsfonts}
\usepackage{amssymb}
\usepackage{amsthm}
\usepackage{graphicx}
\usepackage{color}
\usepackage{amsthm,amssymb}
\usepackage{graphicx}
\usepackage{enumerate}
\usepackage{amssymb,amsmath,graphicx,amsfonts,euscript}
\usepackage{color}
\usepackage{mathrsfs}
\usepackage{empheq}
\usepackage{cases}
\usepackage{cite}

\usepackage[colorlinks=true,backref]{hyperref}
\hypersetup{urlcolor=blue, citecolor=green, linkcolor=blue}

\usepackage[margin=3cm, a4paper]{geometry}

\def\H{{\cal H}}

\def\N{\mathbb{N}}
\def\R{\mathbb{R}}
\def\Z{\mathbb{Z}}
\def\T{\mathbb{T}}
\def\C{\mathbb{C}}

\def\H2{H^2(\R^N)}
\def\L2{L^2(\R^N)}
\def\to{\rightarrow}

\newcommand{\dt}{\,\mathrm{d}t}

 \def\norm#1{\|#1\|}
\def\normb#1{\bigg\|#1\bigg\|} 

\def\H{{\cal H}}

\def\H1{H^1(\R)}

\newcommand{\al}{\alpha}

\newcommand{\re}{\mathop{\mathrm{Re}}}

   \newcommand{\I}{\infty}

 \newcommand{\Del}[1]{}

\numberwithin{equation}{section}

\newtheorem{thm}{Theorem}[section]

\newtheorem{lem}[thm]{Lemma}
\newtheorem{prop}[thm]{Proposition}
\newtheorem{definition}[thm]{Definition}

\theoremstyle{remark}
\newtheorem{remark}[thm]{Remark}
\newtheorem*{exam*}{Examples}
\newtheorem{lemma}{Lemma}

\newcommand{\EQ}[1]{\begin{align*}\begin{split} #1 \end{split}\end{align*}}
\newcommand{\EQn}[1]{\begin{align}\begin{split} #1 \end{split}\end{align}}
\newcommand{\EQnnsub}[1]{\begin{subequations}\begin{align} #1 \end{align}\end{subequations}}
\def\norm#1{\left\|#1\right\|}

\def\normb#1{\big\|#1\big\|}
\def\normbb#1{\Big\|#1\Big\|}

\def\pd{\partial}

\newcommand{\ra}{{\rightarrow}}

\def\lsm{\lesssim}

\newcommand{\sgn}{{\mbox{sgn}}}

\def\al{\alpha}

\begin{document}

	\setcounter{page}{1}

	\title[NLS with rough potential]{Regularization for the Schr\"{o}dinger equation with rough potential: one-dimensional case}

\author{Ruobing Bai}
\address{Ruobing Bai \newline School of Mathematics and Statistics\\
Henan University\\
Kaifeng 475004, China \newline
Center for Applied Mathematics of Henan Province\\
Henan University\\
Zhengzhou 450046, China
}
\email{baimaths@henu.edu.cn}

\thanks{}

	\author{Yajie Lian}
	\address{Yajie Lian \newline Center for Applied Mathematics\\
		Tianjin University\\
		Tianjin 300072, China}
	\email{yjlian@tju.edu.cn}
	\thanks{}

	\author{Yifei Wu}
	\address{Yifei Wu  \newline School of Mathematical Sciences\\
		Nanjing Normal University\\
		Nanjing 210046, China}
	\email{yerfmath@gmail.com}
	\thanks{}
	
	

	\begin{abstract}\noindent
		In this work, we investigate the following Schr\"odinger equation with a spatial potential
		\begin{align*}
			i\partial_t u+\pd_x^2 u+\eta u=0,
		\end{align*}
		where   $\eta$ is a given spatial potential (including the  delta potential and $|x|^{-\gamma}$-potential). Our goal is to provide the regularization mechanism of this model when the potential $\eta\in L_x^r+L_x^\infty$ is rough. In this paper, we mainly focus on one-dimensional case and establish the following results:

1) When the potential $\eta \in L_x^1+L_x^\infty(\R)$, then the solution is in $H_x^{\frac 32-}(\R)$; however, there exists some $\eta \in L_x^1+L_x^\infty(\R)$ such that the solution is not in $H_x^{\frac 32}(\R)$;

2) When the potential $\eta \in L_x^r+L_x^\infty(\R)$ for $1<r\leq 2$, then the solution is in $H_x^{\frac 52-\frac 1r}(\R)$; however, there exists some $\eta \in L_x^r+L_x^\infty(\R)$ such that the solution is not in $H_x^{\frac 52-\frac 1r+}(\R)$;

3) When the potential $\eta \in L_x^r+L_x^\infty(\R)$ for $r>2$, then the solution is in $H_x^{2}(\R)$; however, there exists some $\eta \in L_x^r+L_x^\infty(\R)$ such that the solution is not in $H_x^{2+}(\R)$.

Hence, we provide a complete classification of the regularity mechanism.
Our proof is mainly based on the application of the commutator, local smoothing effect and normal form method. Moreover, a vector-valued Coifman-Meyer multiplier theorem is established, which extends the classical multi-linear multiplier theory to the time-space mixed norm setting. Additionally, we also discuss, without proof, the influence of the existence of nonlinearity on the regularity of solution.
		
	\end{abstract}
	
	\maketitle

%
	
	\section{Introduction}
	\vskip 0.2cm

	In this paper, we study the following linear Schr\"odinger equation with a ``rough" spatial potential
	\begin{equation}\label{eq:NLS}
		\left\{ \aligned
		&i\partial_t u(t, x)+\pd_x^2 u(t, x)+\eta(x) u(t,x)=0,
		\\
		&u(0,x)=u_0(x),
		\endaligned
		\right.
	\end{equation}
	where $u(t, x):\R^+\times \R\rightarrow \C$ is an unknown function,  $\eta:\R\rightarrow \C$ is a given spatial potential.


	The equation \eqref{eq:NLS} has a rich physical background and arises in the mathematical description of phenomena in nonlinear optics and plasma physics. In particular, the equation \eqref{eq:NLS} is often regarded as the disordered Schr\"odinger equation with $\eta(x)$ a given complex-valued random/rough enough potential, which can describe the phenomena known as Anderson localization \cite{Anderson-1958}. The Anderson localization has been widely applied in various fields such as Metal-Insulator Transition, superconductors, suppressing epileptic seizures and so on.

The general form of the equation \eqref{eq:NLS} with a nonlinearty is the following
\begin{equation}\label{eq:NLS-2}
		\left\{ \aligned
		&i\partial_t u(t, x)+\pd_x^2 u(t, x)+\eta(x) u(t,x)=\lambda |u(t,x)|^{p} u(t,x),
		\\
		&u(0,x)=u_0(x),
		\endaligned
		\right.
	\end{equation}
where $u(t, x):\R^+\times \R\rightarrow \C$ and $\lambda \in \R$. The case $\lambda>0$ is referred to the defocusing case, and the case $\lambda<0$ is referred to the focusing case.
	
	
	
In this paper, we aim to study the regularization mechanism of the Schr\"odinger equation when the potential $\eta\in L_x^r+L_x^\infty$ is rough. The regularity of solutions is one of the core issues in the study of the equation \eqref{eq:NLS-2} when the potential is irregular, which reveals how the interplay between nonlinearity, smooth initial conditions, and roughness of the potential influences the localization properties of the system. Moreover, as shown in \cite{M-Wu-Z24}, the regularity properties play a key role in designing and analyzing numerical schemes for approximating solutions, where the smoothness ensures the convergence and accuracy of computational methods.

	With $\eta$ a general spatial random/rough enough potential, there are only a few results of regularition theory for the equation \eqref{eq:NLS-2}. Below, we briefly review some theoretical results of the NLS equation \eqref{eq:NLS-2}. Cazenave \cite{Cazenave--03} proved the equation \eqref{eq:NLS-2} is globally well-posed in $H_x^1(\R^d)$ for small initial data, with the potential $\eta \in L_x^{\infty}(\R^d)$ real-valued when $d\geq 1$. Additionally, Cazenave also established the local well-posedness of the equation \eqref{eq:NLS-2} in $H_x^2(\R^d)$, with the potential $\eta \in L_x^2+L_x^\infty(\R^d)$ when $d\geq 1$. 
For potentials that are stochastic in time but rather regular in space, Bouard and Debussche \cite{Bouard-De-99} studied the stochastic NLS with multiplicative noise and showed that for some subcritical nonlinearities the $L^2(\R^d)$ solution is almost surely global and unique by using the fixed point argument.
	With $\eta$ white noise in space, Debussche and Weber \cite{De-We-19} obtained that the defocusing NLS equation \eqref{eq:NLS-2} has a global solution almost surely in $H^1(\T^2)$ for smooth initial data, and that the focusing NLS equation \eqref{eq:NLS-2} admits the same result under the additional smallness condition, which is based on a renormalization of this equation and the conserved quantities. Subsequently, Debussche and Martin \cite{De-Ma-19} extended these techniques to the subcritical defocusing NLS equation with white noise on the full space $\R^2$, and obtained that if $p<2$ then this equation has a local solution almost surely in some weighted Besov space, and if $p<1$ then the solution is global. Moreover, the interesting work by Babin, Ilyin and Titi \cite{Babin-Ilyin-Titi-2011} established the unconditional well-posedness results for the periodic KdV equation in $\dot{H}^s$, $s\geq 0$, which provided a new insight into regularization mechanisms for nonlinear dispersive partial differential equations (PDEs) in the periodic setting.

	For a typical potential, i.e. $\eta=\delta$, the corresponding NLS equation reads as
	\begin{align*}
		i\partial_t u+\pd_x^2u+\delta u+\lambda|u|^{p} u=0.
	\end{align*}
	The well-posedness of this equation is known only in $H^1$. Specifically, Goodman, Holmes, and Weinstein \cite{Goodman-Holmes-Weinstein-2004} proved this equation with $\lambda>0$ and $p=2$ is globally well-posed in $H^1(\R)$ by using the boundedness of Schr\"{o}dinger wave operator on $H^1(\R)$ (see \cite{Weder-1999}). Later, in \cite{Fu-Oh-Oz-08}, Fukuizumi, Ohta, and Ozawa further proved that this equation with $\lambda>0$ and $0<p<4$ is globally well-posed in $H^1(\R)$ by the Gagliardo-Nirenberg inequality and the conservation laws.  Moreover, it can be shown that the solution of this equation exhibits a shock at the origin. In fact, we can prove that
	 $$u_x(t, 0+) - u_x(t, 0-) = -u(t, 0).$$
This implies that $u\notin H^\frac32$ when the initial data $u_0$ satisfying  $u_0(0)\ne 0$. Therefore, it is naturally questioned whether $H^1$ is the highest regularity achievable for the solution with smooth initial data.

	In the recent work \cite{M-Wu-Z24}, Mauser, Zhao and  the third author considered the space $\widehat{b}^{s,p}$ with the corresponding norm based on the Fourier coefficients of a function $f$ on $\T$, i.e.,
	\EQ{
		\|f\|_{\widehat{b}^{s,p}}=|\widehat{f}_0|+\||k|^s\widehat{f}_k\|_{l_k^p}.
	}
	The authors obtained that when the potential $\eta \in \widehat{b}^{s,p}$ for $s\geq 0$ and $2<p\leq \infty$, then the cubic NLS equation is locally well-posed in $H^{s+\frac 32+\frac 1p-}(\T)$, but ill-posed in $H^{s+\frac 32+\frac 1p}(\T)$ for some $\eta \in \widehat{b}^{s,p}$.
	The endpoint regularity $H^{s+\frac 32+\frac 1p}(\T)$ can be achieved by slightly updating the potential $\eta$ to $W^{s,p'}$.
	Besides, the authors also considered the potential belongs to the Sobolev space $H^s(\T)$ for $s\geq 0$, and obtained the local well-posedness in $H^{s+2}(\T)$, also the ill-posedness in $H^{s+2+}(\T)$ for some given potential. These results are the first sharp well-posedness results for this model.

	Although there are relatively few mathematical results on the regularization mechanism for the NLS equation with rough potential, this topic has attracted the interest of physicists in the field of quantum mechanics.
	In fact, there are some physical insights, such as those discussed in Section 3.4 in \cite{Kumar-2018}, showing that if the potential $\eta$ is the $\delta$ function, then the solution $u\in C^0/C^1$. Similarly, if the potential $\eta$ has a finite jump, for instance $\eta=\sgn x$, then the solution $u\in C^1/C^2$.

	As mentioned above, the authors in \cite{M-Wu-Z24} established the sharp well-posedness results for the equation \eqref{eq:NLS-2} with rough potential on the torus. In this paper, we aim to study the sharp well-posedness/regularization results for this equation on the whole space. Since the resonance set on the whole space is much larger than that in the torus case, we shall adopt a different approach to address the problem posed by the whole space setting.

	To avoid non-essential analysis, we will focus on the linear equation, as the results can be readily extended to the nonlinear case. A detailed discussion of the nonlinear equation is postponed to Section \ref{sec:1.11}.
	

	\subsection{Main results}
	Next, we consider the equation \eqref{eq:NLS} on the whole space $\R$. Before showing our main results, we give the definitions of well-posedness and ill-posedness.
	\begin{definition}[Well-posedness]\label{Def1}
		The well-posedness of a time dependent PDE can be defined as follows: Denote by $C_t(I, X_0)$ the space of continuous functions from the time interval $I$ to the topological space $X_0$. We say that the Cauchy problem is locally well-posed in $C_t(I, X_0)$ if the following properties hold:
		
		(1) For every $u_0\in X_0$, there exists a strong solution defined on a maximal time interval $I=[0, T_{max})$, with $T_{max}\in (0, +\infty]$.

		(2) There exists some auxiliary space $X$, such that strong solution to this problem is unique in $C_t(I, X_0)\cap X$ .

		(3) The solution map $u_0\mapsto u[u_0]$ is continuous from $X_0$ to $X_0$.
	\end{definition}
	
	 When one of the conditions in the above definition violated, we say the Cauchy problem \eqref{eq:NLS} is ill-posed in space $X_0$. In this work, we refer to the violation of the third condition (around zero solution). Then the specific definition of ill-posedness is the following.
	
	\begin{definition}[Ill-posedness]\label{Def2}
		Let $R>0$ and denote
		\EQ{
			B(R):=\{u_0\in \mathcal{S}:\norm{u_0}_{X_0}\leq R\},
		}
		where $\mathcal{S}$ is the Schwartz space. If there exist $R>0$ and some $u_0\in B(R)$ such that for any $T>0$, the solution map $u_0\mapsto u[u_0]$ is discontinuous from $X_0$ to $C([0,T]; X_0)$. Then we say the Cauchy problem is ill-posed in $X_0$.
	\end{definition}

	Now we state our first well-posedness results for the equation \eqref{eq:NLS}.
In what follows, we define the statement that ``when $\eta \in Y_0$ (some spatial function space), then the equation \eqref{eq:NLS} is sharp well-posed in $H_x^{s}$'' to mean that the problem is well-posed in $H_x^{s}$ for any $\eta \in Y_0$, but ill-posed in $H_x^{s+}$ for some given $\eta \in Y_0$.
	\begin{thm}\label{theorem 1}
		The following statements hold:

(1) When $\eta \in L_x^1+L_x^\infty(\R)$, then the equation \eqref{eq:NLS} is sharp globally well-posed in $H_x^{\frac 32-}(\R)$;

 (2) When $\eta \in L_x^r+L_x^\infty(\R)$ with $1<r\leq 2$, then the equation \eqref{eq:NLS} is sharp globally well-posed in $H_x^{\frac52-\frac1r}(\R)$.
	\end{thm}

	From the above theorem, we see that as the integrability of the rough potential $\eta$ reaches $L_x^2+L_x^\infty(\R)$, the regularity of the solution of the equation \eqref{eq:NLS} correspondingly increases to $H_x^{2}(\R)$. However, the equation is ill-posed in $H_x^{2+}(\R)$. A natural question arises: as the integrability of the rough potential $\eta$ continues to improve, will the regularity of the solution also continue to increase accordingly? Our results below give a negative answer to this question.

	\begin{thm}\label{theorem 3}
		When $\eta \in L_x^r+L_x^\infty(\R)$ with $r>2$, then the equation \eqref{eq:NLS} is sharp globally well-posed in $H_x^{2}(\R)$.
	\end{thm}


	\begin{remark}\label{4201}
We make the following remark regarding the above results.
		\begin{enumerate}
			\item
		In Cazenave's book \cite{Cazenave--03}, it is  claimed that if the potential $\eta \in L_x^{1}+L_x^{\infty}$, then the equation \eqref{eq:NLS} is locally well-posed in $H_x^2$, see Corollary 4.8.6 in \cite{Cazenave--03}. However, in the case where $\eta \in L_x^r+L_x^{\infty}$ with $1\leq r<2$, we provide some counterexamples as in the proof of Theorem \ref{theorem 1}, which implies that this assertion is not valid.

			\item
We regard the $\delta$-function as an $L_x^1$-function in the sense that $\delta=\lim_{\varepsilon\to 0}\varepsilon^{-1}\varphi(\frac{x}{\varepsilon})$, for $\varphi \in L_x^1$. Previously, the well-posedness of the equation \eqref{eq:NLS} with potential $\eta=\delta$ was established only in $H^1$, as shown in \cite{Fu-Oh-Oz-08, Goodman-Holmes-Weinstein-2004}. Our result improves the regularity from $H^1$ to $H^{\frac 32-}$ and achieves its optimality.

Moreover, the difference between the cases $r=1$ and $r>1$ is that the endpoint index $\frac52-\frac1r$ can be attained in the latter case.
			Our results align with those in the periodic case, as shown in Theorems 1.1 and 1.3 in \cite{M-Wu-Z24}.
			However, there are significant differences in the arguments used in the proofs for these two cases, see Section \ref{sec:1.2}.
			
			
			%
			
			\item

			The above two theorems imply that, for a fixed rough potential, the regularity of the solution can only reach a certain level. Once the highest achievable regularity is attained, increasing the smoothness of the initial data will not lead to a corresponding increase in the regularity of the solution. Moreover, once the integrability of $\eta$ exceeds $L_x^2+L_x^{\infty}(\R)$, $H_x^2$ is the highest achievable regularity of the solution for smooth data.

				\item
			From our results above and the Sobolev embedding $H^{\frac12+}(\R)\hookrightarrow L^\infty(\R)$, we  observe that if the potential $\eta(x)=\delta(x)$, then the solution $u\in C^{0, \alpha}/C^1, 0\leq \alpha<1$. Similarly, if the potential $\eta$ has a finite jump, then the solution $u\in C^{1, \alpha}/C^{1, \frac 12}, 0\leq \alpha<\frac12$. These observations are consistent with the physical insights in \cite{Kumar-2018} but more refined.

			\item
			
			As a further extension, if $|\nabla|^s\eta\in L_x^1+L_x^{\infty}(\R)$, then the regularity of the solution to \eqref{eq:NLS} reaches $H_x^{s+\frac 32-}(\R)$; if $|\nabla|^s\eta\in L_x^r+L_x^{\infty}(\R), 1<r\le 2$, then the regularity of the solution to \eqref{eq:NLS} can reach $H_x^{s+\frac52-\frac1r}(\R)$; if $|\nabla|^s\eta\in L_x^r+L_x^{\infty}(\R), r> 2$, then the regularity of the solution to \eqref{eq:NLS} can reach $H_x^{s+2}(\R)$.

%
		
		\end{enumerate}
	\end{remark}
	
		
%
%
%
%


	\subsection{A discussion on the effects of nonlinearity} \label{sec:1.11}
Now, we briefly discuss the effect of the existence of nonlinearity on the regularity of the solution to the equation \eqref{eq:NLS}. To be precise, we will present the well-posedness results for the nonlinear Schr\"odinger equation \eqref{eq:NLS-2}.

As we can observe, the term $\eta u$ and the nonlinear term $|u|^{p} u$ interact with each other, influencing the regularity of the solution. On one hand, the rough potential bounds the regularity of the solution from above. On the other hand, the nonlinear terms bound it from below. Their interaction confines the regularity of the solution to a specific domain.
More precisely, for the one-dimensional classical NLS equation,
	\begin{equation}\label{eq:NLS-222}
		\left\{ \aligned
		&i\partial_t u(t, x)+\pd_x^2 u(t, x)=\lambda |u(t,x)|^{p} u(t,x), \quad (t, x)\in \R\times \R,
		\\
		&u(0,x)=u_0(x),
		\endaligned
		\right.
	\end{equation}
the level of this equation is
$
s_c=\frac 12-\frac 2p,
$
in the sense of scaling.
Form the work of Cazenave and Weissler \cite{Cazenave-Weissler-1990}, the equation \eqref{eq:NLS-222} is locally well-posed in $H_x^s(\R)$, for $s\geq s_c$. Therefore, if we consider the one-dimensional nonlinear equation \eqref{eq:NLS-2} in the resolution space where regularity is at least $L_x^{2}$, the nonlinearity has a weaker influence on the regularity of the solution compared to the rough potential.

Next, we summarize the well-posedness results in $H_x^s$ for the equation \eqref{eq:NLS-2} with potential $\eta\in L_x^r+L_x^{\infty}$. We recall that the equation \eqref{eq:NLS-2} is the following
\begin{equation}\label{eq:NLS-2-2}
		\left\{ \aligned
		&i\partial_t u(t, x)+\pd_x^2 u(t, x)+\eta(x) u(t,x)=\lambda |u(t,x)|^{p} u(t,x),
		\\
		&u(0,x)=u_0(x),
		\endaligned
		\right.
	\end{equation}
where the sign of $\lambda$ does not affect the local well-posedness results. For this equation, a combination of the known well-posedness results for the original NLS equation \eqref{eq:NLS-222} and Theorems \ref{theorem 1}, \ref{theorem 3} can derive its well-posedness results. We have the following claim without proof.

{\bf Claim:} Denote the regularity threshold $s_r$ as follows,
$$s_r=\frac 32-, \mbox{ if } r=1; \quad =\frac 52-\frac 1r,  \mbox{ if } 1<r<2;\quad =2, \mbox{ if } r\geq 2. $$
Suppose that
$$
\mbox{max}\{s_c, 0\}\leq s\leq s_r,\quad s<p+1,
$$
then the nonlinear equation \eqref{eq:NLS-2-2} is locally well-posed in $H_x^s(\R)$.

The proof of this claim follows from the fractional chain rule (see e.g. \cite{Visan-2007}), the standard method used in the well-posedness theory for the original NLS equation \eqref{eq:NLS-222}, and the argument presented in this paper. Moreover, if the potential $\eta$ is real-valued, we further assert that the equation \eqref{eq:NLS-2-2} is globally well-posed in the aforementioned space $H_x^s$, as such a potential generally does not influence the global well-posedness in this setting.

	\subsection{The main difficulty and our method} \label{sec:1.2}
	We briefly state the main difficulty and argument in the work.
	In the proof of the global well-posedness for the equation \eqref{eq:NLS} with rough potential, the main difficulty is that we can not take any derivative of the potential function $\eta$. For instance, in the case where $\eta\in L_x^1+L_x^{\infty}(\R)$, we can obtain almost $\frac 32$-order derivative of the solution, but the usual Strichartz' estimates and Kato-Ponce's inequality for $\eta u$ are no longer applicable. Indeed, using the usual Strichartz' estimates, we encounter the following inequality
	\EQ{
		\norm{\int_0^te^{i(t-\rho)\partial_x^2}\langle\nabla \rangle^{\frac 32-}(\eta u)d\rho}_{L_t^{\infty}L_x^2}\lsm \norm{\langle\nabla \rangle^{\frac 32-}(\eta u)}_{L_t^{\frac 43}L_x^1}.
	}
	This inevitably requires taking derivatives of the potential $\eta$. The same difficulty also occured in the study on the torus, see \cite{M-Wu-Z24}.

	A nice approach is to consider the resonant and non-resonant terms of the above integral term in frequency space, as done in the torus case \cite{M-Wu-Z24}. As described earlier, compared with the periodic case, the resonance set on the whole space is larger. In the periodic case, the frequency is discrete, so low frequencies (except for 0 frequency) can be almost removed. However, in the full space case, since the frequency is continuous,
 the resonance is stronger than in the periodic case.

Consequently, it requires us to find new argument to overcome the difficulties. The main techniques used are the commutator method, normal form method, and the local smoothing effect.

	
To be specific, we write
	\EQn{\label{1121}
		D^s\int_0^t  e^{-i\rho\partial_x^2}(\eta e^{i\rho\partial_x^2}v(\rho))d\rho
		=D^{s-\beta}\int_0^t  e^{-i\rho\partial_x^2}\big(\eta\> D^\beta e^{i\rho\partial_x^2}v(\rho)+[D^\beta, \eta] e^{i\rho\partial_x^2}v(\rho)\big)d\rho,
	}
	where $v(t):=e^{-it\partial_x^2}u(t)$ and $s=\frac 32-$, the parameter $\beta:=1-$ is chosen by our needs and $[\cdot, \cdot]$ is the commutator.
	For the first term on the right-hand side of the above equality, after shifting some derivatives to the solution $v$, we can use the local smoothing effect to close the estimates, where $\beta:=1-$ is chosen to match the most regularity we can gain from the local smoothing effect, see remark \ref{521} below. For the second term, we write it in the frequency space as follows,
	\EQ{
		\int_0^t\int_{\xi=\xi_1+\xi_2}e^{i\rho(|\xi|^2-|\xi_2|^2)}
		(|\xi|^{\beta}-|\xi_2|^{\beta})\widehat{\eta}(\xi_1)\widehat{v}(\rho,\xi_2)d\xi_1d\rho.
	}
	We observe that this integral is temporal non-resonant, as the resonant part, which arises form $|\xi|^2-|\xi_2|^2=0$, vanishes.

	Based on the above observation, for the non-resonance part, i.e. $|\xi|^2\neq|\xi_2|^2 $, inspired by differentiation-by-parts used in \cite{Babin-Ilyin-Titi-2011}, we obtain a factor
	\EQ{
		\frac{|\xi|^{\beta}-|\xi_2|^{\beta}}{|\xi|^{2}-|\xi_2|^{2}}\sim \mbox{min}\{|\xi|^{\beta-2}, |\xi_2|^{\beta-2}\},
	}
	which can eliminate the obstruction operator $D^{s-\beta}$ in the front of \eqref{1121}.

Notably, closing the nonlinear estimates faces a key technical gap: classical Coifman-Meyer multiplier theorems only apply to single-space norm settings, with the standard form
	\[
	\|T(f_1,f_2,\dots,f_k)\|_{L_x^p} \le C\|f_1\|_{L_x^{p_1}}\cdots\|f_k\|_{L_x^{p_k}}.
	\]
	These classical results can be trivially extended to  space-time norms, while it is highly nontrivial to generalize them to the time-space mixed framework required in this work. As a result, the classical theory fails to capture the space-time interaction between the rough potential and the solution. To fill this gap, we establish a new vector-valued Coifman-Meyer multiplier theorem, which is tailored to the mixed $L_x^pL_t^r$-norm framework essential for our estimates. The improved vector-valued multiplier estimate reads as
	\[
	\|T(f_1,f_2,\dots,f_k)\|_{L_x^pL_t^q} \le C\|f_1\|_{L_x^{p_1}L_t^{q_1}}\cdots\|f_k\|_{L_x^{p_k}L_t^{q_k}},
	\]
	which provides fundamental support for the subsequent nonlinear analysis.

	\subsection{Organization of the paper} The rest of the paper is organized as follows. In Section 2, we give some basic notations, lemmas that will be used in this paper. The Sections 3 and 4 are devoted to the proof of the well-posedness results for $r=1$ and $1<r\leq 2$ in Theorem \ref{theorem 1}, respectively. In Section 5, we show the proof of Theorem \ref{theorem 3}. In Appendix, we prove the vector-valued Coifman-Meyer multiplier theorem.

	\section{Preliminary}\label{sec:notations}
	
	\subsection{Notations}
	For any $a\in \mathbb{R}$, $a\pm:=a\pm\epsilon$ for arbitrary small $\epsilon>0$. For any $z\in \mathbb{C}$, we define $\mbox{Re}z$ and $\mbox{Im}z$ as the real and imaginary part of $z$, respectively. Denote the commutator $[A, B]$ by $[A, B]f=ABf-BAf$. Denote $\langle\cdot\rangle=(1+|\cdot|^2)^{\frac 12}$ and $D^\alpha  =(-\partial_x^2)^{\frac \alpha 2} $. We write $X \lesssim Y$ or $Y \gtrsim X$ to indicate $X \leq CY$ for
	some constant $C>0$. If $X \leq CY$ and $Y \leq CX$, we write $X\sim Y$. If $X\leq 2^{-5}Y$, denote $X\ll Y$ or $Y\gg X$. Throughout the whole paper, the letter $C$ will denote suitable positive constant that may vary from line to line. Moreover, we use ``\emph{R.H.S} of $(\cdot)$'' to represent the part on the right-hand side of the estimate $(\cdot)$.
	
We use the following norm to denote the sum of two Banach spaces $X_1$ and $X_2$,
\EQ{
\norm{u}_{X_1+X_2}=\inf\{\norm{u_1}_{X_1}+\norm{u_2}_{X_2}: u_1\in X_1, u_2\in X_2, u=u_1+u_2\}.
}
	We also use the following norms to denote the mixed spaces $L_t^qL_x^r(I\times \R)$ and $L_x^r L_t^q( \R\times I)$, that is
	\begin{align*}
		\|u\|_{L_t^qL_x^r(I\times \R)}=\Big(\int_I \|u\|_{L_x^r( \R)}^qdt\Big)^{\frac{1}{q}},
	\end{align*}
	and
	\begin{align*}
		\|u\|_{L_x^rL_t^q( \R\times I)}=\Big(\int_{\R} \|u\|_{L_t^q( I)}^rdx\Big)^{\frac{1}{r}}.
	\end{align*}
	For simplicity, we often write $L_x^{r}L_t^q:=L_x^{r}L_t^q(\R\times I)$, $L_t^{q}L_x^r:=L_t^{q}L_x^r(I\times \R)$ and some similar simplified norm notations for short.

We use $\widehat{f}$ or $\mathscr{F}f$ to denote the Fourier transform of $f$:
$$
\mathscr{F}f (\xi)= \widehat f(\xi)
	= \int_{\R} e^{- i x\cdot \xi} f(x) dx.
$$
We also define
$$
\mathscr{F}^{-1}g (x)= \int_{\R} e^{ i x\cdot \xi} g(\xi) d\xi.
$$
	The Hilbert space $H_x^s(\R)$ is a Banach space of elements such that
	$\mathscr \langle\xi\rangle^s \widehat u  \in L_{\xi}^2(\R)$, and equipped with the norm $\|u\|_{H_x^s}= \|\langle\xi\rangle^s  \widehat{ u}  (\xi)  \|_{L_{\xi}^2}$. We also have an embedding inequality
	that $\|u\|_{H_x^{s_1}}\lesssim \|u\|_{H_x^{s_2}}$ for any $s_1 \leq s_2$, $s_1, s_2\in \R$.  We take a cut-off function $\chi_{a\leq |\cdot| \leq b}(x) \in C_c^{\infty}(\R)$ for $ b>a>\frac 14$ such that
	$$
	\chi_{a\leq |\cdot| \leq b}(x) =\left\{ \aligned
	&1, \quad a\leq |x| \leq b,
	\\
	&0, \quad |x|\leq a-\frac 14 \mbox{ or } |x|\geq b+\frac 14.
	\endaligned
	\right.
	$$
	We also need the usual inhomogeneous Littlewood-Paley decomposition for the dyadic number. We take a cut-off function $\phi(r) \in C_c{([0, +\infty))}$ such that
	$$
	\phi(r)=\left\{ \aligned
	&1, \quad r\leq 1,
	\\
	&0, \quad r\geq 2.
	\endaligned
	\right.
	$$
Next, we give the definition of Littlewood-Paley dyadic projection operator.
For dyadic number $N\in 2^{\mathbb{N}}$, when $N\geq 1$, let $\phi_{\leq N}(r)=\phi(N^{-1}r)$. Then, we define
$\phi_{1}(r):=\phi(r)$, and $\phi_{N}(r)=\phi_{\leq N}(r)-\phi_{\leq \frac N2}(r)$ for any $N\geq 2$. We define the inhomogeneous Littlewood-Paley dyadic operator
$$
f_1=P_{ 1}f:=\mathscr{F}^{-1}(\phi_{ 1}(|\xi|)\widehat{f}(\xi)),
$$
and for any $N\geq 2$,
$$
f_{ N}=P_{ N}f:=\mathscr{F}^{-1}(\phi_{ N}(|\xi|)\widehat{f}(\xi)).
$$
Then, by definition, we have $f=\sum_{N\in 2^{\N}}f_N$. Moreover, we also define the following:
	\begin{align*}
	&f_{\leq N}=P_{\leq N}f:=\mathscr{F}^{-1}(\phi_{\leq N}(|\xi|)\widehat{f}(\xi)),\\
	&f_{ \ll N}=P_{ \ll N}f:=\mathscr{F}^{-1}(\phi_{ \leq N}(2^5|\xi|)\widehat{f}(\xi)),\\
&f_{\lsm N}=P_{\lsm N}f:=\mathscr{F}^{-1}(\phi_{ \leq N}(2^{-5}|\xi|)\widehat{f}(\xi)).
	\end{align*}
	We also define that $f_{\geq N}=P_{\geq N}f:=f-f_{\leq N}$, $f_{\gg N}=P_{\gg N}f:=f-P_{\lsm N}f$, and $f_{\gtrsim N}=P_{\gtrsim N}f:=f-P_{ \ll N}f$.

	Next, we show the Triebel-Lizorkin Spaces $F_{p}^{\alpha, q}$ with the corresponding norm as follows,
	\EQ{
		\|u\|_{F_{p}^{\alpha, q}}=\|u\|_{L_x^p}+\|N^{\alpha}P_Nu\|_{ L_x^p l_{N\in 2^{\mathbb{N}}}^q}.
	}
	For any $1\leq p <\infty$, we define $l_{N}^p=l_{N\in 2^{\mathbb{N}}}^p$ by its norm
	$$
	\|c_N\|_{l_{N\in 2^{\mathbb{N}}}^p}^p:=\sum_{N\in 2^{\mathbb{N}}}|c_N|^p.
	$$
	For $p =\infty$, we define $l_{N}^{\infty}=l_{N\in 2^{\mathbb{N}}}^{\infty}$ by its norm
	$$
	\|c_N\|_{l_{N\in 2^{\mathbb{N}}}^{\infty}}:=\sup_{N\in 2^{\mathbb{N}}}|c_N|.
	$$
	In this paper, we also use the following abbreviations
	$$
	\sum_{N\geq M}:=\sum_{N, M\in 2^{\mathbb{N}}: N\geq M}, \quad \sum_{N\gtrsim M}:=\sum_{N, M\in 2^{\mathbb{N}}: N\geq 2^{-5}M},\mbox{ and } \sum_{N\ll M}:=\sum_{N, M\in 2^{\mathbb{N}}: N\leq 2^{-5}M}.
	$$
Finally, we give the definition of the Schr\"odinger-admissible pair.
Let the pair $(q, r)$ satisfy
		\begin{align*}
			2\leq q, r\leq \infty  ,\quad \frac{2}{q}+\frac{1}{r}=\frac{1}{2},
		\end{align*}
then we say that the pair $(q, r)$ is Schr\"odinger-admissible.

	\subsection{Basic lemmas}

	\quad In this section, we state some preliminary estimates that will be used in our later sections.
	Firstly, we introduce the following Bernstein estimates that will be used frequently.
\begin{lem}[Bernstein estimates\label{lem:Bernstein}]
		For any $1\leq p \leq q \leq \infty$, $s\geq 0$, and $f\in L_x^p(\R^d)$,
		\begin{align*}
			\|P_{\geq N} f\|_{L_x^p(\R^d)}&\lesssim  N^{-s}\||\nabla|^{ s}P_{\geq N} f\|_{L_x^p(\R^d)},\\
			\||\nabla|^{ s}P_{\leq N} f\|_{L_x^p(\R^d)}&\lesssim N^s\|P_{\leq N} f\|_{L_x^p(\R^d)},\\
			\||\nabla|^{\pm s}P_N f\|_{L_x^p(\R^d)}&\sim N^{\pm s}\|P_N f\|_{L_x^p(\R^d)},\\
			\|P_{\leq N} f\|_{L_x^q(\R^d)}&\lesssim  N^{\frac dp-\frac dq}\|P_{\leq N} f\|_{L_x^p(\R^d)},\\
			\|P_N f\|_{L_x^q(\R^d)}&\lesssim  N^{\frac dp-\frac dq}\|P_N f\|_{L_x^p(\R^d)}.
		\end{align*}
	\end{lem}

	\begin{lem}[Schur's test\label{lem:Schur}]
		For any $a>0$, let sequences $\{a_N\}, \{b_N\}\in l_{N\in 2^{\N}}^2$, then we have
		\EQ{
			\sum_{ N\geq N_1}\Big({\frac{N_1}{N}}\Big)^a a_N b_{N_1}\lesssim\|a_N\|_{l_N^2}\|b_N\|_{l_N^2}.
		}
	\end{lem}
Next, we give an elementary estimate which shall be used later.
\begin{lem}\label{lem:I-M}
Let the function $\phi_{\alpha}$ be
		\EQ{
			\phi_{\alpha}(x, y):=|x|^{\alpha}-|y|^{\alpha},
		}
		with $\alpha>0$ and $\phi_{2}(x, y)\neq 0$. Then for any
		$\beta<2$, we have
		\EQn{\label{ph-div}
			\frac{\phi_{\beta}(x, y)}{\phi_{2}(x, y)}\sim {\mbox{min}}\{|x|^{\beta-2}, |y|^{\beta-2}\}.
		}
\end{lem}
\begin{proof}
		When $|x|\gg |y|$ or $|x|\ll |y|$, we have that for any $\alpha >0$,
$$|\phi_{\alpha}(x, y)|\sim \mbox{max}\{|x|^{\alpha},|y|^{\alpha}\}.$$
 Since $\beta-2<0$, we get
		$$	\frac{\phi_{\beta}(x, y)}{\phi_{2}(x, y)}\sim \mbox{min}\{|x|^{\beta-2}, |y|^{\beta-2}\}.$$
		When $|x|\sim |y|$, by the mean value theorem,
		we can easily obtain
		$$\phi_{\alpha}(x, y)\sim |x|^{\alpha -1}(|x|-|y|)\sim |y|^{\alpha -1}(|x|-|y|),\quad\quad \mbox{for any } \alpha >0,$$and thus
		\EQ{
			\frac{\phi_{\beta}(x, y)}{\phi_{2}(x, y)}\sim |x|^{\beta-2} \sim |y|^{\beta-2}.
		}
		This proves \eqref{ph-div}.
\end{proof}
	Next, we recall the well-known Strichartz's estimates.
	\begin{lem}\label{lem:strichartz}
		(Strichartz's estimates, see \cite{Keel-Tao-1998, Cazenave--03, Strichartz-1977, Ginibre-Velo-1985}) Let $I\subset \R$ be a time interval. Let $(q_j, r_j), j=1,2,$ be Schr\"odinger-admissible,
		then the following statements hold:
		\begin{align}\label{1.2222}
			\|e^{it\pd_x^2}f\|_{L_t^{q_j}L_x^{r_j}(I\times{\R})}\lesssim\|f\|_{L^2(\R)};
		\end{align}
		and
		\begin{align}\label{1.222234}
			\Big\|\int_0^t e^{i(t-s)\pd_x^2}F(s)ds\Big\|_{L_t^{q_1}L_x^{r_1}(I\times{\R})}\lesssim \|F\|_{L_t^{q_2'}L_x^{r_2'}(I\times{\R})},
		\end{align}
		where $\frac{1}{q_2}+\frac{1}{q_2'}=\frac{1}{r_2}+\frac{1}{r_2'}=1$.
	\end{lem}
	The next lemma is the smoothing effects.
	\begin{lem}
		(Smoothing effects, see \cite{KenigPonceVega-CPAM-93, LinaresPonce-09}). Let $I\subset \R$ be an interval, including $I=R$. Then\\
		1)
		\EQn{\label{Smooth1}
			\|D^{\frac 12}e^{it\pd_x^2}f \|_{L_x^\infty L_t^2(\R\times I )} \lesssim  \| f  \|_{L_x^2(\R)},
		}
		for all $f\in L_x^2(\R)$; and\\
		2)
		\EQn{\label{Smooth2}
			\Big\|  D^{\frac 12} \int_0^t e^{i(t-t')\pd_x^2} F(x,t')\dt'\Big\|_{L_t^\infty L_x^2(I\times \R)}
			\lesssim           \|F\|_{L_x^1 L_t^2(\R\times I)};
		}
		3)
		\EQn{\label{Smooth3}
			\Big\|  \partial_x \int_0^t e^{i(t-t')\pd_x^2} F(x,t')\dt' \Big\|_{L_x^\infty L_t^2(\R\times I)}
			\lesssim           \|F\|_{L_x^1 L_t^2(\R\times I)},
		}
		for all $F\in L_x^1 L_t^2(\R\times I)$.
	\end{lem}
	
	\begin{remark}\label{521}
		By the estimate \eqref{Smooth3}, and the Littlewood-Paley decomposition, for any $\beta<1$ and any $F\in L_x^1 L_t^2(\R\times I)$, we have
		\EQn{\label{Smooth3-1}
			\Big\|  D^{\beta}P_{\geq 1}\int_0^t e^{i(t-t')\pd_x^2} F(x,t')\dt' \Big\|_{L_x^\infty L_t^2(\R\times I)}
			\lesssim           \|F\|_{L_x^1 L_t^2(\R\times I)}.
		}

	\end{remark}
		We also need the following Littlewood-Paley theory, see the Remark 2.2.2 in \cite{Gra-14}.
	\begin{lem}[Littlewood-Paley theory\label{lem:littlewood-Paley}]
		Let $1<p<\infty$, for any $\alpha\in \R$, we have
		\EQ{
			\|f\|_{F_{p}^{\alpha, 2}}\sim \|\langle\nabla\rangle^{\alpha}f\|_{L_x^p}.
		}
	\end{lem}


	Next, we show the Coifman-Meyer multiplier theory.
	
	\begin{lem}[Multilinear Coifman-Meyer multiplier estimates, see \cite{Co-Me-91}\label{lem:Coifman-Meyer}]
		Let the function $m$ on ${(\R^d)^k}$ be bounded and let $T_m$ be the corresponding $k$-linear multiplier operator on $\R^d$
		\begin{align*}
			T_m(f_1,\cdots, f_k)(x)=\int_{\R^k}m(\xi_1,\cdots,\xi_k)\hat{f_1}(\xi_1)\cdots\hat{f_k}(\xi_k)e^{ix\cdot(\xi_1+\cdots+\xi_k)}d\xi_1\cdots d\xi_k.
		\end{align*}
		If $L$ is sufficiently large and $m$ satisfies
		\begin{align}\label{condition-m}
			\Big|\partial_{\xi_1}^{\al_1}\cdots\partial_{\xi_k}^{\al_k}m(\xi_1,\cdots,\xi_k)\Big|\lesssim_{\al_1,\cdots, \al_k}(|\xi_1|+\cdots+|\xi_k|)^{-(|\al_1|+\cdots+|\al_k|)},
		\end{align}
		for multi-indices $\al_1,\cdots, \al_k$ satisfying $|\al_1|+\cdots+|\al_k|\leq L$. Then, for $1< p<\infty$, $1<p_1, \cdots, p_k\leq \infty$ and $\frac 1p=\frac{1}{p_1}+\cdots+\frac{1}{p_k}$, we have
		\begin{align*}
			\|T_m(f_1,\cdots, f_k)\|_{L_x^p}\leq C\|f_1\|_{L_x^{p_1}}\cdots\|f_k\|_{L_x^{p_k}}.
		\end{align*}
	\end{lem}
The Coifman-Meyer Multiplier Theorem is reduced to the Mihlin-H\"{o}rmander Multiplier
Theorem when $k=1$ and $1<p<\infty$.

We also need the following vector valued multiplier estimates. The proof is provided in the Appendix.
\begin{lem}\label{vector-estimate}
Let the function $m$ on ${(\R^d)^k}$ be defined as in Lemma \ref{lem:Coifman-Meyer}. Let $T_m$ be the $k$-linear multiplier operator on $\R\times \R^d$ with $\R^d$ be the $d$-dimensional Euclidean space,
		\begin{align*}
			T_m(f_1,\cdots, f_k)(t, x)=\int_{\R^k}m(\xi_1,\cdots,\xi_k)\hat{f_1}(t, \xi_1)\cdots\hat{f_k}(t, \xi_k)e^{ix\cdot(\xi_1+\cdots+\xi_k)}d\xi_1\cdots d\xi_k.
		\end{align*}	
		Then, for $1< p, q<\infty$, $1<p_1, \cdots, p_k, q_1, \cdots, q_k < \infty$ and $\frac 1p=\frac{1}{p_1}+\cdots+\frac{1}{p_k}$, $\frac 1q=\frac{1}{q_1}+\cdots+\frac{1}{q_k}$, we have
		\begin{align*}
			\|T_m(f_1,\cdots, f_k)\|_{L_x^pL_t^q}\leq C\|f_1\|_{L_x^{p_1}L_t^{q_1}}\cdots\|f_k\|_{L_x^{p_k}L_t^{q_k}}.
		\end{align*}
	\end{lem}
We shall use the above result with some functions are independent of time variable $t$. In particular, we have the following
	\begin{remark}
Under the same assumptions of $T_m$ and $m$ as in Lemma \ref{vector-estimate}. Suppose that $f_1, \cdots, f_j$ $(j< k)$ are independent of time variable $t$, then for $1<p, p_1, \cdots, p_k, q, q_{j+1}, \cdots, q_k < \infty$ and $\frac 1p=\frac{1}{p_1}+\cdots+\frac{1}{p_k}$, $\frac 1q=\frac{1}{q_{j+1}}+\cdots+\frac{1}{q_k}$,
\begin{align*}
			\|T_m(f_1,\cdots, f_k)\|_{L_x^pL_t^q}\leq C\|f_1\|_{L_x^{p_1}}\cdots\|f_j\|_{L_x^{p_j}}\|f_{j+1}\|_{L_x^{p_{j+1}}L_t^{q_{j+1}}}\|f_k\|_{L_x^{p_k}L_t^{q_k}}.
		\end{align*}
\end{remark}
	
	In order to prove the ill-posedness results for the equation \eqref{eq:NLS}, we need the following lemma.
	\begin{lem}\label{ill-posed}(See \cite{B-T-06}).
		Consider a quantitatively well-posed abstract equation in spaces $D$ and $S$,
		\begin{align*}
			u=L(f)+N_k(u,\ldots, u),
		\end{align*}
		which means for all $f\in D$, $u_1, \ldots ,u_k\in S$ and for some constant $C>0$,
		\begin{align*}
			\|L(f)\|_{S}\leq C\|f\|_{D}, \quad \|N_k(u_1, \ldots ,u_k)\|_{S}\leq C\|u_1\|_{S}\ldots\|u_k\|_{S}.
		\end{align*}
		Here $(D, \|\|_{D})$ is a Banach space with initial data and $(S, \|\|_{S})$ is a Banach space of space-time functions. Define
		\begin{align*}
			A_1(f):=L(f), \quad A_n(f):=\sum_{n_1,\ldots ,n_k\geq 1, n_1+\ldots+n_k=n}N_k(A_{n_1}(f), \ldots ,A_{n_k}(f)), n>1.
		\end{align*}
		Then for some $C_1>0$, all $f, g\in D$ and all $n\geq 1$,
		\begin{align*}
			\|A_n(f)-A_n(g)\|_{S}\leq C_1^n\|f-g\|_{D}(\|f\|_{D}+\|g\|_{D})^{n-1}.
		\end{align*}
	\end{lem}

	\section{The proof of Theorem \ref{theorem 1} with $r=1$}\label{pf of theorem 1}
Next we proceed to the analysis of well-posedness when $\eta\in L_x^r+L_x^{\infty}(\R)$. In the following, we only need to consider $\eta \in L_x^r$. Indeed, for $\eta=\eta_1+\eta_2$, with $\eta_1\in L_x^r$ and $\eta_2\in L_x^{\infty}$, we denote
$$
\Phi_j(u):=\int_0^t e^{i(t-\rho)\pd_x^2}(\eta_j u)d\rho.
$$
Then $\Phi_1(u)$ and $\Phi_2(u)$ are closed in $H_x^{\gamma_*}$ and $H_x^2$, respectively. Here $\gamma_*=\frac 32-, \mbox{ if } r=1; \quad =\frac 52-\frac 1r,  \mbox{ if } 1<r<2; \quad =2, \mbox{ if } r\geq 2. $ These statements shall be proved in the following three sections. Since $\gamma_*\leq 2$, $\Phi_1(u)$ and $\Phi_2(u)$ are both closed in $H_x^{\gamma_*}$.

In this section, we aim to prove that if $\eta\in L_x^1(\R)$, then the equation \eqref{eq:NLS} is sharp globally well-posed in $H_x^{\frac 32-}(\R)$. We only need to consider the positive time direction case, that is $\R^+$, since the $\R^-$ case can be treated in the same way.

\subsection{Local well-posedness in $H_x^{\frac 32-}(\R)$}\label{LWP-1} We firstly give the local well-posedness result and its proof.
\begin{prop}\label{LWP}
	Let $\eta \in L_x^1(\R)$. Then there exists a positive time $T=T(\norm{\eta}_{L_x^1(\R)})$, such that the equation \eqref{eq:NLS} is locally well-posed in $C([0, T);H_x^{\frac 32-}(\R))$.
\end{prop}
	\begin{proof}
In the proofs of the following, we always restrict the variables on $(t,x)\in [0, T)\times\R$. Let $\varepsilon_0$ be a fixed arbitrary small constant, and denote
 \begin{itemize}
 \item $s=\frac 32-\varepsilon_0$,
\item $\beta=1-\frac  {\varepsilon_0 }2$.
\end{itemize}
 Hence, we have $\beta=s-\frac 12+\frac  {\varepsilon_0 }2$.
	We define the auxiliary space $X(I)$ for $I=[0, T)\subset\R^+$ by the following norm,
	\EQ{
		\|u\|_{X(I)}=\|\langle D\rangle^s u\|_{L_t^{\infty}L_x^2(I\times \R)}+\|D^{\beta}u\|_{L_x^{\infty}L_t^2(\R\times I)}.
	}
By Duhamel's formula, we denote the operator $\Phi$ by
		\EQ{
			\Phi (u)=e^{it \partial_x^2}u_0+i\int_{0}^te^{i(t-\rho)\partial_x^2}(\eta u(\rho))d\rho.
		}
		Denote
		\EQ{
			R:=\|u_0\|_{H_x^s(\R)}.
		}
		By Lemma \ref{lem:strichartz} and the smoothing effect \eqref{Smooth1}, there exists a constant $C_0>0$ such that
		\EQn{\label{linear estimate}
			\|e^{it \partial_x^2}u_0\|_{X(I)}\leq C_0 \|u_0\|_{H_x^s}=C_0 R.
		}
		Next, we aim to prove that the operator $\Phi$ is a contraction mapping in the following space
		\EQ{
			B_{R}:=\{u\in C(I; H_x^{s}(\R)):\|u\|_{X(I)}\leq 2C_0R\}.
		}
		For this purpose, we need the following two lemmas.
		\begin{lem}\label{p1}
			Let $I =[0, T)$ and $T<1$. Then there exists a positive constant $C=C(\|\eta \|_{L_x^1})>0$, such that
			\EQ{
				\Big\|D^{\beta}\int_{0}^te^{i(t-\rho)\partial_x^2}(\eta u(\rho))d\rho\Big\|_{L_x^{\infty}L_t^2}
				\leq C  T^{\frac 12}\|u\|_{L_t^{\infty}H_x^s}.
			}
		\end{lem}
	
		\begin{lem}\label{p2}
			Let $I =[0, T)$ and $T<1$. Then there exist positive constants $\theta>0$ and $C=C(\|\eta \|_{L_x^1})>0$, such that for any $N_0\in2^\N$,
			\begin{align*}
				\Big\|\int_{0}^te^{i(t-\rho)\partial_x^2}&(\eta u(\rho))d\rho\Big\|_{L_t^{\infty}H_x^s}
				\leq C \Big(T^{\frac 34}N_0^s+N_0^{-\theta}\Big)\|u\|_{X(I)}.
			\end{align*}
		\end{lem}
		Now, we give the proof of local well-posedness result, assuming that Lemmas \ref{p1} and \ref{p2} hold.
		By Lemma \ref{p1}, we have that for any $u\in B_{R}$,
		\EQn{\label{261}
			\Big\|D^{\beta}\int_{0}^te^{i(t-\rho)\partial_x^2}(\eta u(\rho))d\rho\Big\|_{L_x^{\infty}L_t^2}
			\leq &2CC_0T^{\frac 12}R.
		}
		By Lemma \ref{p2}, we have that for any $u\in B_{R}$,
		\begin{align}\label{263}
			\Big\|\int_{0}^te^{i(t-\rho)\partial_x^2}(\eta u(\rho))d\rho\Big\|_{L_t^{\infty}H_x^s}
			\leq 2C C_0 R \Big(T^{\frac 34}N_0^s+N_0^{-\theta}\Big).
		\end{align}
		First, we take $N_0=N_0(\theta, \|\eta\|_{L_x^1})$ large enough such that
		\EQn{\label{264}
			2C C_0  N_0^{-\theta}\leq \frac 12C_0.
		}
		Then, we take $T=T(N_0)<1$ such that
		\EQn{\label{265}
			2C C_0 \Big(T^{\frac 12}+T^{\frac 34}N_0^s\Big)\leq \frac 12 C_0.
		}
		Collecting the estimates \eqref{linear estimate}-\eqref{265}, we obtain
		\EQ{
			\|\Phi(u)\|_{X(I)}\leq 2C_0R.
		}
		Hence, we have that $\Phi:B_{R}\ra B_{R}$. Therefore, we complete the proof of this proposition by applying the contraction mapping principle.
	\end{proof}

	Next, we give the proof of Lemma \ref{p1}.
	\begin{proof}[\bf{Proof of Lemma \ref{p1}}]
		Applying the high and low frequency decomposition, we have
\begin{align}\label{1210-1}
			\Big\|D^{\beta}\int_{0}^te^{i(t-\rho)\partial_x^2}(\eta u(\rho))d\rho\Big\|_{L_x^{\infty}L_t^2}
			\lesssim& \Big\|D^{\beta}P_{< 1}\int_{0}^te^{i(t-\rho)\partial_x^2}(\eta u(\rho))d\rho\Big\|_{L_x^{\infty}L_t^2}\notag\\
			&+\Big\|D^{\beta}P_{\geq 1}\int_{0}^te^{i(t-\rho)\partial_x^2}(\eta u(\rho))d\rho\Big\|_{L_x^{\infty}L_t^2}.
		\end{align}
		For the first term in \eqref{1210-1}, by the Minkowski and H\"{o}lder inequalities, and Lemma \ref{lem:Bernstein}, we have
		\begin{align}\label{6201}
			\Big\|D^{\beta}P_{<1}\int_{0}^te^{i(t-\rho)\partial_x^2}(\eta u(\rho))d\rho\Big\|_{L_x^{\infty}L_t^2}\notag
			\lesssim & T^{\frac 12}\Big\|D^{\beta}P_{< 1}\int_{0}^te^{i(t-\rho)\partial_x^2}(\eta u(\rho))d\rho\Big\|_{L_{t,x}^{\infty}}\notag\\
			\lesssim &  T^{\frac 12}\Big\|\int_{0}^te^{i(t-\rho)\partial_x^2}(\eta u(\rho))d\rho\Big\|_{L_t^{\infty}L_x^2}.
		\end{align}
		Further, noting that $s>\frac12$, using Strichartz's estimates and the Sobolev inequality, we get
		\begin{align}\label{1210-2}
			\emph{R.H.S}~ \mbox{of}~ \eqref{6201}
			\lesssim &  T^{\frac 12}\|\eta u\|_{L_t^{\frac 43}L_x^1}\notag\\
			\lesssim&  T^{\frac 54}\|\eta \|_{L_x^1}\|u\|_{L_{t,x}^{\infty}}\notag\\
			\lesssim&  T^{\frac 54}\|\eta \|_{L_x^1}\|u\|_{L_t^{\infty}H_x^s}.
		\end{align}
For the second term in \eqref{1210-1}, noting that $\beta<1$, by the smoothing effect \eqref{Smooth3-1}, and the Sobolev inequality, we have
		\begin{align}\label{1210-3}
			\Big\|D^{\beta}P_{\geq 1}\int_{0}^te^{i(t-\rho)\partial_x^2}(\eta u(\rho))d\rho\Big\|_{L_x^{\infty}L_t^2}&\lesssim \|\eta u \|_{L_x^1L_t^2}\notag\\
			&\lesssim T^{\frac 12}\|\eta \|_{L_x^1}\|u\|_{L_{t,x}^{\infty}}\notag\\
			&\lesssim T^{\frac 12}\|\eta \|_{L_x^1}\|u\|_{L_t^{\infty}H_x^s}.
		\end{align}
		By the estimates \eqref{1210-1}-\eqref{1210-3}, and $T^{\frac 54}< T^{\frac 12}$ for $T<1$, we finish the proof of this lemma.
	\end{proof}
	
	Now, we are in the position to give the proof of Lemma \ref{p2}.
	\begin{proof}[\bf{Proof of Lemma \ref{p2}}]
		By Lemma \ref{lem:strichartz} and Sobolev's inequality, we have
		\EQn{\label{266}
			\Big\|\int_{0}^te^{i(t-\rho)\partial_x^2}(\eta u(\rho))d\rho\Big\|_{L_t^{\infty}L_x^2} \lesssim& \|\eta u\|_{L_t^{\frac 43}L_x^1}\lesssim T^{\frac 34}\|\eta \|_{L_x^1}\|u\|_{L_t^{\infty}H_x^s}.
		}
		Next, by high and low frequency decomposition, we have
		\begin{align}\label{271}
			\Big\|D^s\int_{0}^te^{i(t-\rho)\partial_x^2}(\eta u(\rho))d\rho\Big\|_{L_t^{\infty}L_x^2} \lesssim&\Big\|D^sP_{< N_0}\int_{0}^te^{i(t-\rho)\partial_x^2}(\eta u(\rho))d\rho\Big\|_{L_t^{\infty}L_x^2}\notag\\
			&+\Big\|D^sP_{\geq N_0}\int_{0}^te^{i(t-\rho)\partial_x^2}(\eta u(\rho))d\rho\Big\|_{L_t^{\infty}L_x^2},
		\end{align}
		where $N_0 \in 2^{\N}$.
		For this first term in \eqref{271}, noting that $s>0$, by the same way in \eqref{266}, and Lemma \ref{lem:Bernstein}, we conclude that
		\EQn{\label{277}
			\Big\|D^sP_{< N_0}\int_{0}^te^{i(t-\rho)\partial_x^2}(\eta u(\rho))d\rho\Big\|_{L_t^{\infty}L_x^2}\lesssim T^{\frac 34}N_0^s\|\eta \|_{L_x^1}\|u\|_{L_t^{\infty}H_x^s}.
		}
		For the second term in \eqref{271},  we use the following transform
		\EQ{
			v(t):=e^{-it\partial_x^2}u(t).
		}
		Then we have that
		$$
		\|v\|_{H_x^s}=\|u\|_{H_x^s},
		$$
		and
		\EQ{
			v(t)=u_0+i\int_0^te^{-i\rho\partial_x^2}(\eta e^{i\rho\partial_x^2}v(\rho))d\rho.
		}
		The latter implies that
		\begin{align}
			&\partial_tv=ie^{-it\partial_x^2}(\eta e^{it\partial_x^2}v(t))\label{def-partialtv}.
		\end{align}
		Now we use the commutator to write
		\begin{align}\label{1029-1}
			&D^sP_{\geq N_0}\int_{0}^te^{-i\rho\partial_x^2}(\eta u(\rho))d\rho\notag\\
			=&D^{s-\beta}P_{\geq N_0}\Big(\int_0^t  e^{-i\rho\partial_x^2}(\eta D^\beta e^{i\rho\partial_x^2}v(\rho))d\rho
			+\int_0^t  e^{-i\rho\partial_x^2}[D^\beta, \eta] e^{i\rho\partial_x^2}v(\rho)d\rho\Big)\notag\\
			:=&I+II.
		\end{align}
		Hence, for the second term in \eqref{271},  it reduces to
		\begin{align}\label{1210-2017}
			\Big\|D^sP_{\geq N_0}\int_{0}^te^{i(t-\rho)\partial_x^2}(\eta u(\rho))d\rho\Big\|_{L_t^{\infty}L_x^2}
			=&\Big\|D^sP_{\geq N_0}\int_{0}^te^{-i\rho\partial_x^2}(\eta u(\rho))d\rho\Big\|_{L_t^{\infty}L_x^2}\notag\\
			\le & \big\|I\big\|_{L_t^{\infty}L_x^2}+ \big\| II \big\|_{L_t^{\infty}L_x^2}.
		\end{align}
		Next, we estimate the terms  $I$ and $II$ above one by one.
		For $I$, noting $s-\beta-\frac 12=-\frac {\varepsilon_0}{2}<0$, by Lemmas \ref{lem:Bernstein}, \ref{lem:strichartz}, the smoothing effect \eqref{Smooth2}, and H\"{o}lder's inequality, we have
		\begin{align}\label{278}
			\|I\|_{L_t^{\infty}L_x^2}=&\Big\|P_{\geq N_0}D^{s-\beta}\int_0^t  e^{-i\rho\partial_x^2}(\eta D^\beta u(\rho))d\rho\Big\|_{L_t^{\infty}L_x^2}\notag\\
			\lesssim& N_0^{s-\beta-\frac 12}\|\eta D^{\beta}u\|_{L_x^1L_t^2}\notag\\
			\lesssim& N_0^{s-\beta-\frac 12}\|\eta\|_{L_x^1} \|D^{\beta}u\|_{L_x^{\infty}L_t^2}.
		\end{align}
		For $II$, we use the normal form argument.
		By the Fourier transform, and integration-by-parts, we have
		\begin{align}\label{4321}
			\widehat{II}(\xi)=&\int_0^t\int_{\xi=\xi_1+\xi_2}\chi_{\geq N_0}(\xi)|\xi|^{s-\beta}e^{i\rho\phi_2(\xi, \xi_2)}
			\phi_{\beta}(\xi, \xi_2)\widehat{\eta}(\xi_1)\widehat{v}(\rho,\xi_2)d\xi_1d\rho\notag\\
			=&-i\int_{\xi=\xi_1+\xi_2}\chi_{\geq N_0}(\xi)|\xi|^{s-\beta}e^{i\rho\phi_2(\xi, \xi_2)}\frac{\phi_{\beta}(\xi, \xi_2)}{\phi_{2}(\xi, \xi_2)}\widehat{\eta}(\xi_1)\widehat{v}(\rho,\xi_2)d\xi_1\big|_{\rho=0}^{\rho=t}\notag\\
			&-\int_0^t\int_{\xi=\xi_1+\xi_2+\xi_3}\chi_{\geq N_0}(\xi)|\xi|^{s-\beta}e^{i\rho\phi_{2}(\xi, \xi_3)}
			\frac{\phi_{\beta}(\xi, \xi_2+\xi_3)}{\phi_{2}(\xi, \xi_2+\xi_3)}\widehat{\eta}(\xi_1)\widehat{\eta}(\xi_2)\widehat{v}(\xi_3)d\xi_1d\xi_2d\rho\notag\\
			:=&\widehat{II}_1(\xi)+\widehat{II}_2(\xi).
		\end{align}
		Noting that we ignore the case $\phi_2(\xi, \xi_2)=0$ in \eqref{4321}, since $\phi_{\beta}(\xi, \xi_2)=0$ in this case.
		
		For the term $\widehat{II}_1(\xi)$, by Lemma \ref{lem:I-M}, the Plancherel identity and H\"{o}lder's inequality, we have
		\begin{align}\label{279}
			\|II_1\|_{L_t^{\infty}L_x^2}=&\|\widehat{II}_1(\xi)\|_{L_t^{\infty}L_{\xi}^2}\notag\\
			\lesssim &\sup_t\sup_{h:\|h\|_{L_x^2}\leq 1}\int_{\xi=\xi_1+\xi_2}\chi_{\geq N_0}(\xi)|\xi|^{s-\beta}\mbox{min}\{|\xi|^{\beta-2}, |\xi_2|^{\beta-2}\}|\widehat{\eta}(\xi_1)||\widehat{v}(\xi_2)||\widehat{h}(\xi)|d\xi_1d\xi\notag\\
			\lesssim &\sup_t\sup_{h:\|h\|_{L_x^2}\leq 1}\int\chi_{\geq N_0}(\xi)|\xi|^{s-2}|\widehat{h}(\xi)|\langle\xi_2\rangle^{-s}\langle\xi_2\rangle^{s}|\hat{v}(\xi_2)|d\xi_2d\xi\|\widehat{\eta}\|_{L_{\xi}^{\infty}}\notag\\
			\lesssim & \sup_{h:\|h\|_{L_x^2}\leq 1}\|\chi_{\geq N_0}(\xi)|\xi|^{s-2}\|_{L_{\xi}^2}\|\widehat{h}(\xi)\|_{L_{\xi}^2}\|\langle\xi\rangle^{s}|\hat{v}(\xi)|\|_{L_t^{\infty}L_{\xi}^2}\|\widehat{\eta}\|_{L_{\xi}^{\infty}}\notag\\
			\lesssim & N_0^{s-\frac 32}\|\eta\|_{L_x^1}\|v\|_{L_t^{\infty}H_x^s}.
		\end{align}
		Next, we consider the term $\widehat{II_2}(\xi)$ and claim that
		\begin{align}\label{27622}
			\|II_2\|_{L_t^{\infty}L_x^2}\lesssim  T\|\eta\|_{L_x^{1}}^2\|v\|_{L_t^{\infty}H_x^s}+N_0^{s-\frac 52}\|\eta\|_{L_x^{1}}^2\|u\|_{L_t^{\infty}H_x^s}
			+T\|\eta\|_{L_x^{1}}^3\|u\|_{L_t^{\infty}H_x^s}.
		\end{align}
		First of all, by the high and low frequency decomposition, we have that
		\begin{align}\label{631}
			\widehat{II}_{2}(\xi)&=
			-\int_0^t\int_{\xi=\xi_1+\xi_2+\xi_3}\chi_{\geq N_0}(\xi)|\xi|^{s-\beta}e^{i\rho\phi_{2}(\xi, \xi_3)}
			\frac{\phi_{\beta}(\xi, \xi_2+\xi_3)}{\phi_{2}(\xi, \xi_2+\xi_3)}\widehat{\eta}(\xi_1)\widehat{\eta}(\xi_2)\widehat{v}(\xi_3)d\xi_1d\xi_2d\rho\notag\\
			&=-\int_0^t\int_{\substack{\xi=\xi_1+\xi_2+\xi_3\\ |\xi|\sim |\xi_3|}}\chi_{\geq N_0}(\xi)|\xi|^{s-\beta}e^{i\rho\phi_{2}(\xi, \xi_3)}
			\frac{\phi_{\beta}(\xi, \xi_2+\xi_3)}{\phi_{2}(\xi, \xi_2+\xi_3)}\widehat{\eta}(\xi_1)\widehat{\eta}(\xi_2)\widehat{v}(\xi_3)d\xi_1d\xi_2d\rho\notag\\
			&\quad-\int_0^t\int_{\substack{\xi=\xi_1+\xi_2+\xi_3\\ |\xi|\ll |\xi_3|\mbox{ \footnotesize{or} }|\xi|\gg |\xi_3|}}\chi_{\geq N_0}(\xi)|\xi|^{s-\beta}e^{i\rho\phi_{2}(\xi, \xi_3)}
			\frac{\phi_{\beta}(\xi, \xi_2+\xi_3)}{\phi_{2}(\xi, \xi_2+\xi_3)}\widehat{\eta}(\xi_1)\widehat{\eta}(\xi_2)\widehat{v}(\xi_3)d\xi_1d\xi_2d\rho\notag\\
			&:=\widehat{II}_{21}(\xi)+\widehat{II}_{22}(\xi).
		\end{align}
		For $II_{21}$ in \eqref{631}, setting $\widetilde{\xi}_2=\xi_2+\xi_3$, by Lemma \ref{lem:I-M} and the Littlewood-Paley decomposition, we have
		\begin{align}\label{272}
			\|\widehat{II}_{21}(\xi)\|_{L_t^{\infty}L_{\xi}^2}
			\lesssim &\sup_t\sup_{h:\|h\|_{L_x^2}\leq 1}\int_0^t\int_{\substack{\xi=\xi_1+\xi_2+\xi_3\\ |\xi|\sim |\xi_3|}}|\xi|^{s-\beta}\mbox{min}\{|\xi|^{\beta-2}, |\widetilde{\xi}_2|^{\beta-2}\}\notag\\
			&\quad\quad\cdot|\widehat{\eta}(\xi_1)||\widehat{\eta}(\xi_2)||\widehat{v}(\xi_3)||\widehat{h}(\xi)|d\xi_1d\xi_2d\xi d\rho\notag\\
			\lesssim &\sup_t\sup_{h: \|h\|_{L_x^2}\leq 1}\sum_{j}\sum_{|k-j|\leq 5}\int_0^t\int_{\R^3}|\xi|^{s-\beta}\mbox{min}\{|\xi|^{\beta-2}, |\widetilde{\xi}_2|^{\beta-2}\}\notag\\
			&\quad\quad\cdot|\widehat{P_{2^k}v}(\xi_3)||\widehat{P_{2^j}h}(\xi)|d\widetilde{\xi}_2d\xi_3d\xi d\rho\|\widehat{\eta}\|_{L_{\xi}^{\infty}}^2.
		\end{align}
		Moreover, for any $\gamma<-1$, we have
		\begin{align}\label{6202}
			\int_{\R}\mbox{min}\{|\xi|^{\gamma}, |\widetilde{\xi}_2|^{\gamma}\}d\widetilde{\xi}_2\lesssim |\xi|^{\gamma+1}.
		\end{align}
		Hence, by H\"{o}lder's inequality, Lemma \ref{lem:littlewood-Paley} and  \eqref{6202}, we have
		\begin{align}\label{II21}
			\|\widehat{II}_{21}(\xi)\|_{L_t^{\infty}L_{\xi}^2}
			\lesssim&\sup_t\sup_{h:\|h\|_{L_x^2}\leq 1}\sum_{j}\sum_{|k-j|\leq 5}\int_0^t\int_{\R^2}|\xi|^{s-1}|\widehat{P_{2^k}v}(\xi_3)|\notag\\
			&\quad\quad\cdot|\widehat{P_{2^j}h}(\xi)|d\xi_3d\xi d\rho\>\|\widehat{\eta}\|_{L_{\xi}^{\infty}}^2\notag\\
			\lesssim&\sup_t\sup_{h:\|h\|_{L_x^2}\leq 1}\sum_{j}\sum_{|k-j|\leq 5}T2^{j(s-1)+\frac j2+\frac k2}\|\widehat{P_{2^k}v}\|_{L_{\xi}^2}\|\widehat{P_{2^j}h}\|_{L_{\xi}^2}\|\widehat{\eta}\|_{L_{\xi}^{\infty}}^2\notag\\
			\lesssim&\sup_{h:\|h\|_{L_x^2}\leq 1}T\|\eta\|_{L_x^{1}}^2\|2^{js}\widehat{P_{2^j}v}\|_{L_t^{\infty}l_j^2L_{\xi}^2}\|\widehat{P_{2^j}h}\|_{l_j^2L_{\xi}^2}\notag\\
			\lesssim& T\|\eta\|_{L_x^{1}}^2\|v\|_{L_t^{\infty}H_x^s}.
		\end{align}
		Next, we consider $II_{22}$ in \eqref{631}. Firstly, by integration-by-parts, we have
		\begin{align}\label{4322}
			\widehat{II}_{22}(\xi)=&-\int_0^t\int_{\substack{\xi=\xi_1+\xi_2+\xi_3\\ |\xi|\ll |\xi_3|\mbox{ \footnotesize{or} }|\xi|\gg |\xi_3|}
			}\chi_{\geq N_0}(\xi)
			|\xi|^{s-\beta}e^{i\rho\phi_{2}(\xi, \xi_3)}
			\frac{\phi_{\beta}(\xi, \xi_2+\xi_3)}{\phi_{2}(\xi, \xi_2+\xi_3)}\notag\\
			&\quad\quad\quad\quad\cdot\widehat{\eta}(\xi_1)
			\widehat{\eta}(\xi_2)\widehat{v}(\rho,\xi_3)d\xi_1d\xi_2d\rho\notag\\
			=&i\int_{\substack{\xi=\xi_1+\xi_2+\xi_3\\ |\xi|\ll |\xi_3|\mbox{ \footnotesize{or} }|\xi|\gg |\xi_3|}}\chi_{\geq N_0}(\xi)|\xi|^{s-\beta}e^{i\rho\phi_{2}(\xi, \xi_3)}
			\frac{\phi_{\beta}(\xi, \xi_2+\xi_3)}{\phi_{2}(\xi, \xi_3)\phi_{2}(\xi, \xi_2+\xi_3)}\notag\\
			&\quad\quad\quad\quad\cdot\widehat{\eta}(\xi_1)\widehat{\eta}(\xi_2)\widehat{v}(\rho,\xi_3)d\xi_1d\xi_2\Big|_{\rho=0}^{\rho=t}\notag\\
			&+\int_0^t\int_{\substack{\xi=\xi_1+\xi_2+\xi_3\\ |\xi|\ll |\xi_3+\xi_4|\mbox{ \footnotesize{or} }|\xi|\gg |\xi_3+\xi_4|}}\chi_{\geq N_0}(\xi)|\xi|^{s-\beta}e^{i\rho\phi_{2}(\xi, \xi_4)}
			\frac{\phi_{\beta}(\xi, \xi_2+\xi_3+\xi_4)}{\phi_{2}(\xi, \xi_2+\xi_3+\xi_4)}\notag\\
			&\quad\quad\quad\quad\cdot\frac{1}{\phi_{2}(\xi, \xi_3+\xi_4)}\widehat{\eta}(\xi_1)\widehat{\eta}(\xi_2)\widehat{\eta}(\xi_3)\widehat{v}(\rho,\xi_4)
			d\xi_1d\xi_2d\xi_3d\rho\notag\\
			:=&\widehat{II}_{221}(\xi)+\widehat{II}_{222}(\xi).
		\end{align}
		Under the frequency restriction of $|\xi|\ll |\xi_3|$ or $|\xi|\gg |\xi_3|$, we have
		\begin{align}\label{45671}
			\frac{1}{|\phi_{2}(\xi, \xi_3)|}\thicksim \mbox{min}\{|\xi|^{-2},|\xi_3|^{-2}\} .
		\end{align}
		Hence, for the boundary term $\widehat{II}_{221}(\xi)$ in \eqref{4322},  by \eqref{ph-div}, \eqref{45671}, Lemma \ref{lem:I-M}, and using the variable substitution: $\widetilde{\xi}_2:=\xi_2+\xi_3$, we have
		\begin{align}\label{273}
			\|\widehat{II}_{221}(\xi)\|_{L_t^{\infty}L_{\xi}^2}
			\lesssim&\sup_t\sup_{h:\|h\|_{L_x^2}\leq 1}\int_{\xi=\xi_1+\xi_2+\xi_3}\chi_{\geq N_0}(\xi)|\xi|^{s-\beta-2}\mbox{min}\{|\xi|^{\beta-2},|\xi_2+\xi_3|^{\beta-2}\}\notag\\
			&\quad\quad\cdot|\widehat{\eta}(\xi_1)||\widehat{\eta}(\xi_2)||\widehat{v}(t,\xi_3)||\widehat{h}(\xi)|d\xi_1d\xi_2d\xi\notag\\
			\lesssim&\sup_t\sup_{h:\|h\|_{L_x^2}\leq 1}\int_{\R^3}\chi_{\geq N_0}(\xi)|\xi|^{s-\beta-2}\mbox{min}\{|\xi|^{\beta-2},|\widetilde{\xi}_2|^{\beta-2}\}\notag\\
			&\quad\quad\cdot|\widehat{v}(t,\xi_3)||\widehat{h}(\xi)|d\widetilde{\xi}_2d\xi_3d\xi\> \|\widehat{\eta}\|_{L_{\xi}^{\infty}}^2.
		\end{align}
		Noting that for $s=\frac 32-$, we have the following inequality,
		\begin{align}\label{6305}
			\int_{\R}|\widehat{v(t)}(\xi)|d\xi=&\int_{\R}\langle\xi\rangle^{-s}\langle\xi\rangle^{s}|\widehat{v(t)}(\xi)|d\xi\notag\\
			\lesssim &\|\langle\xi\rangle^{-s}\|_{L_{\xi}^2}\|\langle\xi\rangle^{s}\widehat{v}(\xi)\|_{L_{\xi}^2}\notag\\
			\lesssim &\|v\|_{L_t^{\infty}H_x^s}.
		\end{align}
		Further, noting that $\beta-2<-1$, by H\"{o}lder's inequality, \eqref{6202}, \eqref{273} and \eqref{6305}, we have
		\begin{align}\label{43167}
			\|\widehat{II}_{221}(\xi)\|_{L_t^{\infty}L_{\xi}^2}
			\lesssim&\sup_t\sup_{h:\|h\|_{L_x^2}\leq 1}\int_{\R^2}\chi_{\geq N_0}(\xi)|\xi|^{s-3}|\widehat{v}(\xi_3)||\widehat{h}(\xi)|d\xi_3d\xi\>\|\widehat{\eta}\|_{L_{\xi}^{\infty}}^2\notag\\
			\lesssim&\sup_{h:\|h\|_{L_x^2}\leq 1}\|\chi_{\geq N_0}(\xi)|\xi|^{s-3}\|_{L_{\xi}^2}\|\widehat{h}(\xi)\|_{L_{\xi}^2}\|v\|_{L_t^{\infty}H_x^s}\|\widehat{\eta}\|_{L_{\xi}^{\infty}}^2\notag\\
			\lesssim& N_0^{s-\frac 52}\|\eta\|_{L_x^{1}}^2\|v\|_{L_t^{\infty}H_x^s}.
		\end{align}
		Next, for the integral term $\widehat{II}_{222}(\xi)$  in \eqref{4322}, using the variable substitution: $\widetilde{\xi}_2=\xi_2+\xi_3+\xi_4$ and $\widetilde{\xi}_3=\xi_3+\xi_4$, by \eqref{6202}, \eqref{6305}, and Lemma \ref{lem:I-M}, we obtain
		\begin{align}\label{274}
			\|\widehat{II}_{222}(\xi)\|_{L_t^{\infty}L_{\xi}^2}\lesssim& T\sup_t\sup_{h:\|h\|_{L_x^2}\leq 1} \int_{|\xi|\geq N_0}|\xi|^{s-\beta}\mbox{min}\{|\xi|^{\beta-2},|\widetilde{\xi}_2|^{\beta-2}\}\mbox{min}\{|\xi|^{-2},|\widetilde{\xi}_3|^{-2}\}\notag\\
			&\quad\quad\cdot|\widehat{v}(t,\xi_4)||\widehat{h}(\xi)|d\widetilde{\xi}_2d\widetilde{\xi}_3d\xi_4d\xi \>\|\widehat{\eta}\|_{L_{\xi}^{\infty}}^3\notag\\
			\lesssim& T\sup_t\sup_{h:\|h\|_{L_x^2}\leq 1} \int_{|\xi|\geq N_0}|\xi|^{s-2}|\widehat{v}(\xi_4)||\widehat{h}(\xi)|d\xi_4d\xi\|\eta\|_{L_x^{1}}^3\notag\\
			\lesssim& T\|\eta\|_{L_x^{1}}^3\|v\|_{L_t^{\infty}H_x^s}.
		\end{align}

		Hence, collecting the estimates \eqref{631}, \eqref{II21}, \eqref{4322}, \eqref{43167} and \eqref{274}, we obtain the claim \eqref{27622}.

		Therefore, by the estimates \eqref{266}-\eqref{277}, \eqref{1029-1}, \eqref{278}, \eqref{279}, and \eqref{27622}, we finish the proof of Lemma \ref{p2}.
	\end{proof}
	
\subsection{Global well-posedness in $H_x^{\frac 32-}(\R)$}\label{GWP} We are now in a position to prove the global well-posedness.
	\begin{proof}[\bf{Proof}]
Let $u\in C([0,T^*);H_x^{\frac 32-}(\R))$ be the solution of equation \eqref{eq:NLS} with the maximal lifespan $[0,T^*)$.

Let $0<\epsilon_0<T$, where $T=T( \norm{\eta}_{L_x^{1}})$ is the lifespan obtained in the above subsection. Assume by contradiction that $T^*<+\I$. Using the argument in the proof of the local well-posedness, we conclude that $u\in C([0,T^*-\epsilon_0);H_x^{\frac 32-}(\R))$ and $\norm{u(T^*-\epsilon_0)}_{H_x^{\frac 32-}} \lsm \norm{u_0}_{H_x^{\frac 32-}} $.

Hence, using the argument in the proof of the local well-posedness again, we can further extend solution $u$ beyond $T^*$. To be precise, we obtain that $u\in C([0,T^*-\epsilon_0+T);H_x^{\frac 32-}(\R))$. We see $T^*-\epsilon_0+T>T^*$, this contradicts to the definition of $T^*$. Therefore, this proves that $T^*=+\infty$.
\end{proof}

\subsection{Ill-posedness in $H_x^{s}(\R)$, $s\geq \frac 32$} 	Next, we prove that for any $s\geq \frac 32$, there exists some $\eta\in L_x^1(\R)$, such that the equation \eqref{eq:NLS} is ill-posed in $H_x^{s}(\R)$. The main tool is Lemma \ref{ill-posed}.
	\begin{proof}[\bf{Proof}]
		We only need to show the ill-posedness in $H_x^{\frac 32}$. First of all, let $f=f(x)$ be a time-independent function, and define
\EQ{
			A [f]&=\int_0^t e^{-is\partial_x^2}(\eta e^{is\partial_x^2}f)ds.
		}
To achieve our goal, it is sufficient to show that for given $\eta$, we have that for any $T>0$ and $M_0>0$, there exists $u_0\in \mathcal{S}$, such that
\begin{align}\label{1211-1}
\sup\limits_{t\in [0,T]}\|A[u_0](t)\|_{H_x^{\frac 32}}\geq M_0.
\end{align}
Next, on one hand, we choose the initial data $u_0\in \mathcal{S}$ such that
		$$u_0=P_{\leq 1}u_0, \quad \widehat{u}_0\geq 0, \mbox{ and } \int_{\R} \widehat{u}_0(\xi)\,d\xi>0.$$
		(For example, $u_0(x):=P_{\leq 1}e^{-|x|^2}$ satisfies the conditions above).

		On the other hand, we choose the potential
		\EQ{
			\eta=\varepsilon^{-1}\varphi(\frac x {\varepsilon}),
		}
where $\varphi(x)=e^{-|x|^2}$, and $\varepsilon>0$ is a fixed arbitrary small constant.
		Then we have
		\EQ{
			\|\eta\|_{L_x^1}=\|\varphi\|_{L_x^1}\lsm 1, \mbox{ and } \widehat{\eta}(\xi)=e^{-\varepsilon^2|\xi|^2}.
		}
Moreover, let $N_0$ be a large constant determined later, and
		\EQ{
			t:=\frac 1{N_0},
		}
		and the set
		\EQ{
			\Omega:=\bigcup_{k=N_0}^{N_0^2}\Omega_k:=\bigcup_{k=N_0}^{N_0^2}\Big(\sqrt{N_0}\sqrt{2k\pi+\frac{5\pi}{12}}, \sqrt{N_0}\sqrt{2k\pi+\frac{\pi}{2}}\Big),
		}
		where $\Omega_k\cap \Omega_j=\phi$ if $k\neq j$.

		For $A [u_0]$, by the Fourier transform and the choice of $\eta$, we have
		\EQ{
			\widehat{A [u_0]}(\xi)=&\int_0^t \int_{\xi=\xi_1+\xi_2}e^{is(|\xi|^2-|\xi_2|^2)}\widehat{\eta}(\xi_1)\widehat{u_0}(\xi_2)d\xi_2ds\\
			=&\int_0^t \int_{\xi=\xi_1+\xi_2}e^{is(|\xi|^2-|\xi_2|^2)}e^{-\varepsilon^2|\xi_1|^2}\widehat{u_0}(\xi_2)d\xi_2ds.
		}
		Let $|\xi|\geq N_0$ and noting that $|\xi_2|\leq 1$ by the definition of $u_0$, we have
		\EQ{
			\widehat{A [u_0]}(\xi)
			=&\int_{\xi=\xi_1+\xi_2}\frac{e^{it(|\xi|^2-|\xi_2|^2)}-1}{i(|\xi|^2-|\xi_2|^2)}e^{-\varepsilon^2|\xi_1|^2}\widehat{u}_0(\xi_2)d\xi_2.
		}
Take the real part of $\widehat{A [u_0]}(\xi)$, we get
\EQn{\label{reA1010}
		\re	\big(\widehat{A [u_0]}(\xi)\big)
			=\int_{\xi=\xi_1+\xi_2}\frac{\sin[t(|\xi|^2-|\xi_2|^2)]}{|\xi|^2-|\xi_2|^2}e^{-\varepsilon^2|\xi_1|^2}\widehat{u}_0(\xi_2)d\xi_2.
		}
By the mean value theorem, we have
\EQ{
\sin[t(|\xi|^2-|\xi_2|^2)]=\sin (t|\xi|^2)+O(t|\xi_2|^2).
}
Noting that if $\xi \in \Omega$, then  $t|\xi|^2\in (2k\pi+\frac{5\pi}{12}, 2k\pi+\frac{\pi}{2})$, which further implies $\sin (t|\xi|^2)\ge \frac 12$. Moreover, for $N_0$ large enough, we have $t|\xi_2|^2=\frac 1 {N_0^2} \leq \frac 1 4$. Hence, we conclude that
\begin{align}\label{1010}
\sin[t(|\xi|^2-|\xi_2|^2)]\geq \frac 14.
\end{align}	
Moreover, taking sufficiently large $N_0$ such that $N_0^{\frac 32}\gg \frac 1{\varepsilon}$, we get
$$
\varepsilon^2|\xi_1|^2\sim \varepsilon^2|\xi|^2\ll 1.
$$
Hence, this derives the following
\begin{align}\label{e-1010}
e^{-\varepsilon^2|\xi_1|^2}\sim e^{-\varepsilon^2|\xi|^2}\gtrsim 1.
\end{align}
By the estimates \eqref{reA1010}-\eqref{e-1010}, we obtain
\EQn{\label{reA1010-1}
		\re	\big(\widehat{A [u_0]}(\xi)\big)
			\geq &C\int_{\R}\frac{1}{|\xi|^2-|\xi_2|^2}\widehat{u}_0(\xi_2)d\xi_2\\
\geq &C\int_{\R}\frac{1}{|\xi|^2}\widehat{u}_0(\xi_2)d\xi_2.
		}	
	Noting that $\mbox{Re}\big(\widehat{A [u_0]}(\xi)\big)>0$, the above inequality yields that
\begin{align*}
\|A [u_0]\|_{H_x^{\frac 32}(\R)}= \norm{\langle\xi\rangle^{\frac 32}\widehat{A[u_0]}(\xi)}_{L_{\xi}^2(\R)}\geq \norm{\langle\xi\rangle^{\frac 32}\mbox{Re}\big(\widehat{A[u_0]}(\xi)\big)}_{L_{\xi}^2(\R)}.
\end{align*}	
Hence, we conclude that
\begin{align}\label{es-A}
\|A [u_0]\|_{H_x^{\frac 32}(\R)}^2\geq \int_{\R}\big|\langle\xi\rangle^{\frac 32}\mbox{Re}\big(\widehat{A[u_0]}(\xi)\big)\big|^2d\xi
\geq  C_0\int_{\Omega}\frac {1}{|\xi|}d\xi,
		\end{align}
		where $C_0:=(\int \widehat{u}_0(\xi)d\xi)^2>0$.
		
		By $\Omega_k\cap \Omega_j=\phi$ if $k\neq j$, we have
		\EQ{
			\int_{\Omega}\frac {1}{|\xi|}d\xi=&\sum_{k=N_0}^{N_0^2}\int_{\sqrt{N_0}\sqrt{2k\pi+\frac{5\pi}{12}}}^{\sqrt{N_0}\sqrt{2k\pi+\frac{\pi}{2}}}\frac {1}{|\xi|}d\xi\\
			=&\frac 12\sum_{k=N_0}^{N_0^2}\ln\frac {2k\pi+\frac{\pi}{2}}{2k\pi+\frac{5\pi}{12}}.
		}
		Noting that $\ln(1+x)\geq \frac 12x$ when $0\leq x\leq  1$, thus taking $N_0$ large enough, for any $k\geq N_0$, we conclude that
		\EQ{
			\ln\frac {2k\pi+\frac{\pi}{2}}{2k\pi+\frac{5\pi}{12}}\geq\frac {\pi}{24}\cdot \frac {1}{2k\pi+\frac{5\pi}{12}}.
		}
		Further, by the above two estimates, we have
		\begin{align}\label{6311}
			\int_{\Omega}\frac {1}{|\xi|}d\xi\geq & \frac {\pi}{48} \sum_{k=N_0}^{N_0^2} \frac {1}{2k\pi+\frac{5\pi}{12}}\notag\\
			\geq & \frac {\pi}{48}\ln \frac {2N_0^2\pi+\frac{5\pi}{12}}{2N_0\pi+\frac{5\pi}{12}}\notag\\
			\geq & \frac {\pi}{50}\ln N_0.
		\end{align}
		Hence, by the estimates \eqref{es-A} and \eqref{6311}, we obtain \eqref{1211-1}.
		Therefore, the proof of ill-posedness is done by applying Lemma \ref{ill-posed}.
	\end{proof}

\section{The proof of Theorem \ref{theorem 1} with $1<r\leq 2$}\label{local-r}

\subsection{Resonant and non-resonant decomposition}
	
First of all, we introduce the technique of the resonant and non-resonant decomposition based on the normal form method introduced by Shatah \cite{Sha85CPAM}, which shall be used in the proof of global well-posedness when $\eta \in L_x^r(\R)$ for $r>1$.
	
	By Duhamel's formula, the integral equation for \eqref{eq:NLS} is
	\EQn{\label{Duhamel-NLS}
		u(t)=e^{it\partial_x^2}u_0+i\int_{0}^t e^{i(t-\rho)\partial_x^2}(\eta u)(\rho)d\rho.
	}
	Next, we apply the normal form transform to give a suitable resonant and non-resonant decomposition for the integral term in \eqref{Duhamel-NLS}. Firstly, we give the following definition.
	
	\begin{definition}
		\label{defn:normal-form}
		Let $N_0\in 2^\N$ be constant, for any $s\in \R$, denote the multiplier
		\EQ{
			m(\xi_1, \xi_2):=\frac{\langle\xi\rangle^s\langle\xi_1\rangle^{2-s}}{|\xi|^2-|\xi_2|^2}\phi_{\geq N_0}(|\xi|)\phi_{\ll 1}\Big(\frac{|\xi_2|}{|\xi|}\Big), {\mbox{ with }}  \xi=\xi_1+\xi_2.
		}
		Using this notation, we give the following definitions:
		\begin{enumerate}
			\item (Boundary term)
			We define the normal form transform for functions $f, g$ by
			\EQ{
				\mathcal{B}(f, g)(x) := \int_{\xi=\xi_1+\xi_2} e^{ix(\xi_1+\xi_2)} m(\xi_1, \xi_2)  \widehat{f}(\xi_1)\widehat{g}(\xi_2)d\xi_1d\xi_2.
			}
			\item (Resonance term and low frequency term)
			Next, we also define the resonance part and some remainder terms of the nonlinear term by
			\EQ{
				\mathcal{R}(\eta, u):= P_{\leq N_0}(\eta u)+P_{\geq N_0}\sum_{M\gtrsim N}P_N(\eta P_Mu).
			}
		\end{enumerate}
	\end{definition}
	\begin{remark}
	It is easy to check that the multiplier $m$ satisfies the conditions of Coifman-Meyer's multiplier in Lemma \ref{lem:Coifman-Meyer}.
	\end{remark}
	Using the notations in the above definition, we can rewrite $\langle\nabla \rangle^{s} u(t,x)$ in the following form.
	\begin{lem}\label{resoant-decom} Let $u(t,x)$ be defined in \eqref{Duhamel-NLS}, the bilinear operator $\mathcal{B}$ and the function $\mathcal{R}(\eta, u)$ be defined in Definition \ref{defn:normal-form}. Then for any $s\in \R$, we have
		\EQn{\label{Duhamel-NLS-normal}
			\langle\nabla \rangle^{s} u(t,x)=&\langle\nabla \rangle^{s}e^{it\partial_x^2}u_0(x)
			-e^{it\partial_x^2}\mathcal{B}(\langle\nabla \rangle^{-2+s}\eta, u_0(x))\\
&+\mathcal{B}(\langle\nabla \rangle^{-2+s}\eta, u(t,x))\\
			&+i\int_0^te^{i(t-\rho)\partial_x^2}\langle\nabla \rangle^{s}\mathcal{R}(\eta, u(\rho, x))d\rho\\
			&-i\int_0^te^{i(t-\rho)\partial_x^2}\mathcal{B}(\langle\nabla \rangle^{-2+s}\eta, \eta u)(\rho, x)d\rho.
		}
	\end{lem}
	\begin{proof}
		First of all, using the high-low frequency decomposition, we have
		\EQn{\label{Duhamel-NLS-1}
			\langle\nabla \rangle^{s}u(t)=\langle\nabla \rangle^{s}e^{it\partial_x^2}u_0+i\int_{0}^t \langle\nabla \rangle^{s}e^{i(t-\rho)\partial_x^2}\big(P_{\leq N_0}(\eta u)+P_{\geq N_0}(\eta u)\big)(\rho)d\rho,
		}
		where $N_0\in 2^\N$ is a large enough constant.
		
		Next, we consider the integral term involving $P_{\geq N_0}(\eta u)$. For convenience, we denote
		\EQ{
			I\triangleq \int_{0}^t \langle\nabla \rangle^{s}e^{i(t-\rho)\partial_x^2}P_{\geq N_0}(\eta u)d\rho.
		}
		By the Fourier transform, we have
		\EQn{\label{6115091}
			\widehat{I}(\xi)=&\int_{0}^t \langle\xi \rangle^{s}e^{-i(t-\rho)|\xi|^2}\phi_{\geq N_0}(|\xi|)\widehat{\eta u}(\xi)d\rho\\
			=&\int_{0}^t \int_{\xi=\xi_1+\xi_2}\langle\xi \rangle^{s}e^{-i(t-\rho)|\xi|^2}\phi_{\geq N_0}(|\xi|)\widehat{\eta}(\xi_1)\widehat{u}(\xi_2)d\xi_1d\rho\\
			=&\int_{0}^t \int_{\xi=\xi_1+\xi_2}\langle\xi \rangle^{s}e^{-i(t-\rho)|\xi|^2}\phi_{\geq N_0}(|\xi|)\phi_{\gtrsim 1}\Big(\frac{|\xi_2|}{|\xi|}\Big)\widehat{\eta}(\xi_1)\widehat{u}(\xi_2)d\xi_1d\rho\\
			&+\int_{0}^t \int_{\xi=\xi_1+\xi_2}\langle\xi \rangle^{s}e^{-i(t-\rho)|\xi|^2}\phi_{\geq N_0}(|\xi|)\phi_{\ll 1}\Big(\frac{|\xi_2|}{|\xi|}\Big)\widehat{\eta}(\xi_1)\widehat{u}(\xi_2)d\xi_1d\rho\\
			\triangleq &\widehat{I_1}(\xi)+\widehat{I_2}(\xi).
		}
		For $I_1$, we have that
		\EQn{\label{611509}
			I_1=\sum_{M\gtrsim N}\int_{0}^t\langle\nabla \rangle^{s}P_{\geq N_0}e^{i(t-\rho)\partial_x^2}P_N(\eta P_Mu)d\rho.
		}
		Next, for $I_2$. Let $u=e^{it\partial_x^2}v$, then
		\EQ{
			\widehat{I_2}(\xi)=e^{-it|\xi|^2}\int_{0}^t \int_{\xi=\xi_1+\xi_2}\langle\xi \rangle^{s}e^{i\rho(|\xi|^2-|\xi_2|^2)}\phi_{\geq N_0}(|\xi|)\phi_{\ll 1}\Big(\frac{|\xi_2|}{|\xi|}\Big)\widehat{\eta}(\xi_1)\widehat{v}(\xi_2)d\xi_1d\rho.
		}
		Due to this term is non-resonant, we can use the integration-by-parts to treat it. Here, we note that
		\EQ{
			\partial_{\rho}\widehat{v}(\xi_2)=ie^{i\rho|\xi_2|^2}\widehat{\eta u}(\xi_2).
		}
		Hence, we have
		\EQ{
			\widehat{I_2}(\xi)=& \int_{\xi=\xi_1+\xi_2}\langle\xi \rangle^{s}\frac{e^{-it|\xi_2|^2}}{i(|\xi|^2-|\xi_2|^2)}\phi_{\geq N_0}(|\xi|)\phi_{\ll 1}\Big(\frac{|\xi_2|}{|\xi|}\Big)\widehat{\eta}(\xi_1)\widehat{v}(\xi_2)d\xi_1\\
			&-e^{-it|\xi|^2} \int_{\xi=\xi_1+\xi_2}\langle\xi \rangle^{s}\frac{1}{i(|\xi|^2-|\xi_2|^2)}\phi_{\geq N_0}(|\xi|)\phi_{\ll 1}\Big(\frac{|\xi_2|}{|\xi|}\Big)\widehat{\eta}(\xi_1)\widehat{v_0}(\xi_2)d\xi_1\\
			&-\int_0^t\int_{\xi=\xi_1+\xi_2}\langle\xi \rangle^{s}\frac{e^{-i(t-\rho)|\xi|^2}}{|\xi|^2-|\xi_2|^2}\phi_{\geq N_0}(|\xi|)\phi_{\ll 1}\Big(\frac{|\xi_2|}{|\xi|}\Big)\widehat{\eta}(\xi_1)\widehat{\eta u}(\xi_2)d\xi_1d\rho.
		}
		Using the notation of the multiplier $m(\xi_1, \xi_2)$, we can rewrite the above identity further as follows
		\EQ{
			\widehat{I_2}(\xi)=&-i\int_{\xi=\xi_1+\xi_2}m(\xi_1, \xi_2)\langle\xi_1 \rangle^{-2+s}\widehat{\eta}(\xi_1)\widehat{u}(\xi_2)d\xi_1\\
			&+ie^{-it|\xi|^2} \int_{\xi=\xi_1+\xi_2}m(\xi_1, \xi_2)\langle\xi_1 \rangle^{-2+s}\widehat{\eta}(\xi_1)\widehat{u_0}(\xi_2)d\xi_1\\
			&-\int_0^t\int_{\xi=\xi_1+\xi_2}e^{-i(t-\rho)|\xi|^2}m(\xi_1, \xi_2)\langle\xi_1 \rangle^{-2+s}\widehat{\eta}(\xi_1)\widehat{\eta u}(\xi_2)d\xi_1d\rho.
		}
		Using the definition of bilinear operator $\mathcal{B}$ and Fourier inverse transform, we get
		\EQn{\label{611553}
			I_2=&-i\mathcal{B}(\langle\nabla \rangle^{-2+s}\eta, u(t,x))+ie^{it\partial_x^2}\mathcal{B}(\langle\nabla \rangle^{-2+s}\eta, u_0(x))\\
			&-\int_0^te^{i(t-\rho)\partial_x^2}\mathcal{B}(\langle\nabla \rangle^{-2+s}\eta, \eta u)(\rho, x)d\rho.
		}
		Collecting the estimates \eqref{Duhamel-NLS-1}-\eqref{611553}, we finish the proof of this lemma.
		
\end{proof}

\subsection{Global well-posedness in $H_x^{\frac 52-\frac 1r}(\R)$}\label{local-r}
	In this part, we give the proof that if $\eta \in L_x^r(\R)$ for $1<r\leq 2$, then the equation \eqref{eq:NLS} is globally well-posed in $H_x^{\frac52-\frac1r}(\R)$. For the proof of the global well-posedness, the strategy is to apply Lemma \ref{resoant-decom} and give the estimates on \eqref{Duhamel-NLS-normal} term by term.
	Next, we firstly give the necessary estimates to prove the global well-posedness.

	\subsubsection{Boundary terms}
	\begin{lem}[Boundary terms]\label{lem:nonlinear-estimate-boundary-1d}
		Let $1<r\le 2$, $s=\frac 52-\frac 1r$,  and $I \subset\R^+$ be an interval containing $0$. Then, for any $N_0\in2^\N$,
		\EQn{\label{32511-1}
			\normb{e^{it\partial_x^2}\mathcal{B}(\langle\nabla \rangle^{-2+s}\eta, u_0)}_{L_t^{\infty}L_x^2(I\times \R)} \lsm \|P_{\geq N_0}\eta\|_{L_x^{r}}\|u\|_{L_t^{\infty}H_x^{s}},
		}
		and
		\EQn{\label{32512-1}
			\norm{\mathcal{B}(\langle\nabla \rangle^{-2+s}\eta, u(t))}_{L_t^{\infty}L_x^2(I\times \R)} \lsm  \|P_{\geq N_0}\eta\|_{L_x^{r}}\|u\|_{L_t^{\infty}H_x^{s}}.
		}
	\end{lem}
	\begin{proof}
		Using Strichartz's estimates, Lemma \ref{lem:Coifman-Meyer} and Sobolev's inequality, we obtain
		\begin{align}\label{32511-11}
			\normb{&e^{it\partial_x^2}\mathcal{B}(\langle\nabla \rangle^{-2+s}\eta, u_0)}_{L_t^{\infty}L_x^2(I\times \R)}\notag \\
			\lsm& \|P_{\geq N_0} \langle\nabla\rangle^{-2+s}\eta\|_{L_x^{2}}\|u_0\|_{L_x^{\infty}}\notag\\
			\lsm &\|P_{\geq N_0} \langle\nabla\rangle^{-2+s+\frac 1r-\frac 12}\eta\|_{L_x^{r}}\|u\|_{L_{t, x}^{\infty}}\notag\\
			\lsm&\|P_{\geq N_0} \eta\|_{L_x^{r}}\|u\|_{L_t^{\infty}H_x^{s}},
		\end{align}
		where we used the condition $s>\frac 12$ and $-2+s+\frac 1r-\frac 12=s-\frac 52+\frac 1r= 0$. This gives \eqref{32511-1}. For \eqref{32512-1}, in the same way as above, we can get it. Hence, we complete the proof of the lemma.
	\end{proof}
	
	\subsubsection{Resonance term and low frequency term}
	\begin{lem}\label{lem:nonlinear-estimate-resonance-1d}
		Let $1<r\le 2$, $s=\frac 52-\frac 1r$,  and $I=[0, T) \subset\R^+$  be an interval. Then, for any $N_0\in2^\N$,
		\EQ{
			\normbb{\int_{0}^t e^{i(t-\rho)\partial_x^2}\langle\nabla \rangle^{s}\mathcal{R}(\eta, u) d\rho}_{L_t^{\infty}L_x^2(I\times \R)} \lsm T^{\frac 12}N_0^{s+\frac 12}\|\eta\|_{L_x^{r}}\|u\|_{L_t^{\infty}H_x^{s}}.
		}
	\end{lem}
	\begin{proof}
		Recalling that
		$
		\mathcal{R}(\eta, u):= P_{\leq N_0}(\eta u)+P_{\geq N_0}\sum_{M\gtrsim N}P_N(\eta P_Mu).
		$
		Then we have
		\EQnnsub{
			\Big\|\int_{0}^t& e^{i(t-\rho)\partial_x^2}\langle\nabla \rangle^{s}\mathcal{R}(\eta, u) d\rho\Big\|_{L_t^{\infty}L_x^2}\notag\\
			\lsm & \Big\|\int_{0}^t e^{i(t-\rho)\partial_x^2}\langle\nabla \rangle^{s} P_{\leq N_0}(\eta u)d\rho\Big\|_{L_t^{\infty}L_x^2}\label{63002}\\
			&+\Big\|\int_{0}^t \sum_{M\gtrsim N}e^{i(t-\rho)\partial_x^2}\langle\nabla \rangle^{s} P_{\geq N_0}P_N(\eta P_Mu)d\rho\Big\|_{L_t^{\infty}L_x^2}\label{63003}.
		}
		For the term \eqref{63002}, using the Strichartz estimates and Bernstein estimates,
		\begin{align}\label{63001}
			\eqref{63002}\lsm&\|\langle\nabla \rangle^{s} P_{\leq N_0}(\eta u)\|_{L_t^{1}L_x^2}\notag\\
			\lsm & N_0^{s+\frac 12}\|\eta u\|_{L_{t, x}^{1}}\notag\\
			\lsm & TN_0^{s+\frac 12}\|\eta\|_{L_x^r}\|u\|_{L_{t}^{\infty}H_x^{s}}.
		\end{align}
		For the term \eqref{63003}, by the duality formula, Strichartz's estimates, smoothing effect \eqref{Smooth2} and  Lemma \ref{lem:Schur}, we have
		\begin{align}\label{32514-1}
			\eqref{63003}
			\lsm& \sup_{\norm{h}_{L_x^2\leq 1}}\sum _{N\lsm M}\normbb{\Big\langle\int_{0}^t\langle\nabla \rangle^{s} e^{i(t-\rho)\Delta}P_N(\eta P_Mu )d\rho, h\Big\rangle}_{L_t^{\infty}}\notag\\
			\lsm &\sup_{\norm{h}_{L_x^2\leq 1}}\sum _{N\lsm M}\frac{\langle N \rangle^{s-\frac 12}}{\langle M \rangle^{s-\frac 12}}\normbb{\int_{0}^t e^{i(t-\rho)\Delta}\langle\nabla \rangle^{\frac 12}P_N(\eta \langle M \rangle^{s-\frac 12}P_Mu )d\rho}_{L_t^{\infty}L_x^2}\|P_N h\|_{L_x^2}\notag\\
			\lsm &\sup_{\norm{h}_{L_x^2\leq 1}}\sum _{N\lsm M}\frac{\langle N \rangle^{s-\frac 12}}{\langle M \rangle^{s-\frac 12}} \norm{\eta \langle M \rangle^{s-\frac 12}P_Mu}_{L_x^1 L_t^2}\|P_N h\|_{L_x^2}\notag\\
			\lsm &\norm{\eta \langle M \rangle^{s-\frac 12}P_Mu}_{l_M^2L_x^1L_t^2}.
		\end{align}
		Denote $r'=\frac r{r-1}\in [2, +\infty)$, by the H\"{o}lder, Minkowski and Sobolev inequalities, and  Lemma \ref{lem:littlewood-Paley}, we get
		\begin{align}\label{32515-1}
			\norm{\eta \langle M \rangle^{s-\frac 12}P_Mu}_{l_M^2L_x^1L_t^2}\lsm& \|\eta\|_{L_x^r}\|\langle M \rangle^{s-\frac 12}P_Mu\|_{L_x^{r'}L_t^2l_M^2}\notag\\
			\lsm& \|\eta\|_{L_x^r}\|\langle M \rangle^{s-\frac 12}P_Mu\|_{L_t^2 L_x^{r'}l_M^2}\notag\\
			\lsm & \|\eta\|_{L_x^r}\|u\|_{L_t^2F_{r'}^{s-\frac 12, 2}}\notag\\
			\lsm &T^{\frac 12}\|\eta\|_{L_x^r}\|\langle \nabla \rangle^{s-\frac 12}u\|_{L_t^{\infty}L_x^{r'}}\notag\\
			\lsm &T^{\frac 12}\|\eta\|_{L_x^r}\|\langle \nabla \rangle^{s}u\|_{L_t^{\infty}L_x^{2}}.
		\end{align}
		Further, by the above two estimates,
		\begin{align}\label{63006}
			\eqref{63003}\lsm T^{\frac 12}\|\eta\|_{L_x^r}\|\langle \nabla \rangle^{s}u\|_{L_t^{\infty}L_x^{2}}.
		\end{align}
		Collecting the estimates \eqref{63001} and \eqref{63006}, we finish the proof of this lemma.
	\end{proof}
	\subsubsection{High-order terms}
	\begin{lem}[High-order terms]\label{lem:nonlinear-estimate-higher-order-1d}
		Let $1<r\le 2$, $s=\frac 52-\frac 1r$,  and $I=[0, T) \subset\R^+$ be an interval with $T<1$. Then
		\EQ{
			\normbb{\int_0^te^{i(t-\rho)\partial_x^2}\mathcal{B}(\langle\nabla \rangle^{-2+s}\eta, \eta u)d\rho}_{L_t^{\infty}L_x^2(I\times \R)} \lsm& T^{\frac 12}\|\eta\|_{L_x^r}^2\|u\|_{L_t^{\infty}H_x^{s}}.
		}
	\end{lem}
	\begin{proof}
		By the smoothing effect \eqref{Smooth2}, we have
	\begin{align}\label{25-1201}
			\normbb{\int_0^te^{i(t-\rho)\partial_x^2}\mathcal{B}(\langle\nabla \rangle^{-2+s}\eta, \eta u)d\rho}_{L_t^{\infty}L_x^2(I\times \R)}
			\lsm \norm{\langle\nabla \rangle^{-\frac 12}\mathcal{B}(\langle\nabla \rangle^{-2+s}\eta, \eta u)}_{L_x^1L_t^2}.
		\end{align}
By the Strichartz estimates and the Minkowski inequality, we have
\begin{align}\label{25-1202}
			\normbb{\int_0^te^{i(t-\rho)\partial_x^2}\mathcal{B}(\langle\nabla \rangle^{-2+s}\eta, \eta u)d\rho}_{L_t^{\infty}L_x^2(I\times \R)}
			\lsm \norm{\mathcal{B}(\langle\nabla \rangle^{-2+s}\eta, \eta u)}_{L_{x,t}^{\frac65}}.
		\end{align}
Fixing $\epsilon_0$ satisfying $0<\epsilon_0<\frac 32(1-\frac 1r)$, by \eqref{25-1201}, \eqref{25-1202}, and the interpolation, we have
\begin{align}\label{25-1203}
			\normbb{\int_0^te^{i(t-\rho)\partial_x^2}\mathcal{B}(\langle\nabla \rangle^{-2+s}\eta, \eta u)d\rho}_{L_t^{\infty}L_x^2(I\times \R)}
			\lsm \norm{\langle\nabla \rangle^{-\frac 12(1-\epsilon_0)}\mathcal{B}(\langle\nabla \rangle^{-2+s}\eta, \eta u)}_{L_x^{r_0}L_t^{q_0}},
		\end{align}
where $r_0$ and $q_0$ satisfy $\frac 1{r_0}=1-\frac {\epsilon_0}{6}$ and $\frac 65<q_0<2$.

Next, we take $r_1$ satisfying $\frac 1{r_1}=\frac 1{r_0}-\frac 1r$, then $-\frac 52+s+\frac {\epsilon_0}{2}+\frac 1r-\frac 1{r_1+}=\frac 23\epsilon_0-1+\frac 1r\leq 0$.
Hence, by Lemma \ref{vector-estimate}, the Sobolev and Minkowski inequalities, we get
		\begin{align}\label{eq:nonlinear-estimate-higher-order-211}
		\norm{\langle\nabla \rangle^{-\frac 12(1-\epsilon_0)}\mathcal{B}(\langle\nabla \rangle^{-2+s}\eta, \eta u)}_{L_x^{r_0}L_t^{q_0}}\lsm & \|\langle\nabla \rangle^{-\frac 52+s+\frac {\epsilon_0}{2}} P_{\geq N_0}\eta\|_{L_x^{r_1+}}\|\eta\|_{L_x^r}\|u\|_{L_x^{\infty-}L_t^{q_0}}\notag\\
			\lsm &\|\langle\nabla \rangle^{-\frac 52+s+\frac {\epsilon_0}{2}+\frac 1r-\frac 1{r_1+}}P_{\geq N_0}\eta\|_{L_x^r}\|\eta\|_{L_x^r}\|u\|_{L_t^{q_0}L_x^{\infty-}}\notag\\
			\lsm &T^{\frac 1{q_0}}\|\eta\|_{L_x^r}^2\|u\|_{L_t^{\infty}H_x^{s}}.
		\end{align}
	Finally, noting that $T^{\frac 1{q_0}}<T^{\frac 1{2}}$ for $T<1$,	this gives the proof of this lemma.
	\end{proof}
		
Based on the above several lemmas, we are now in a position to prove the global well-posedness.

	\begin{proof}[\bf{Proof}]
	Recall that $1<r\le 2$, $s=\frac 52-\frac 1r$ and let $I=[0, T)\subset \R^+$. First of all, by Strichartz's estimate, we have
		\EQn{\label{446}
			\norm{e^{it\pd_x^2} u_0}_{L_t^{\infty}H_x^s(I\times \R)}=\|u_0\|_{H_x^{s}}:=R.
		}
		Fixing $0<\delta\ll 1$, by $\eta \in L_x^r(\R)$ for $1<r\leq 2$, we take $N_0=N_0(\delta)\in 2^{\N}$, such that
		\EQn{\label{447}
			\|P_{\geq N_0} \eta\|_{L_x^{r}}\leq \delta.
		}
		Denote the operator $\Phi$ by the following form,
		\EQ{
			\langle\nabla\rangle^s\Phi(u)=&\langle\nabla \rangle^{s}e^{it\partial_x^2}u_0(x)
			-e^{it\partial_x^2}\mathcal{B}(\langle\nabla \rangle^{-2+s}\eta, u_0(x))\\
&+\mathcal{B}(\langle\nabla \rangle^{-2+s}\eta, u(t,x))\\
			&+i\int_0^te^{i(t-\rho)\partial_x^2}\langle\nabla \rangle^{s}\mathcal{R}(\eta, u(\rho, x))d\rho\\
			&-i\int_0^te^{i(t-\rho)\partial_x^2}\mathcal{B}(\langle\nabla \rangle^{-2+s}\eta, \eta u)(\rho, x)d\rho.
		}
		Taking the working space as
		\EQ{
			B_{R}:=\{u\in C(I; H_x^{s}(\R)):\|u\|_{L_t^{\infty}H_x^s(I\times \R)}\leq 2R\}.
		}
		Next, we aim to prove $\Phi$ is the contraction mapping in $B_{R}$. Hence, we need to collect the estimates of $\langle\nabla\rangle^s\Phi(u)$ in $L_t^{\infty}L_x^2$.

		By Lemma \ref{lem:nonlinear-estimate-boundary-1d},
		\EQn{\label{442}
			\normb{e^{it\partial_x^2}\mathcal{B}(\langle\nabla \rangle^{-2+s}\eta, u_0)}_{L_t^{\infty}L_x^2(I\times \R)}\lsm \delta R,
		}
		and
		\EQn{\label{443}
			\norm{\mathcal{B}(\langle\nabla \rangle^{-2+s}\eta, u(t))}_{L_t^{\infty}L_x^2(I\times \R)}\lsm \delta R.
		}
		By Lemma \ref{lem:nonlinear-estimate-resonance-1d},
		\EQn{\label{444}
			\normbb{\int_{0}^t e^{i(t-\rho)\partial_x^2}\langle\nabla \rangle^{s}\mathcal{R}(\eta, u) d\rho}_{L_t^{\infty}L_x^2(I\times \R)} \lsm& T^{\frac 12}RN_0^{s+\frac 12}\|\eta\|_{L_x^r}.
		}
		By Lemma \ref{lem:nonlinear-estimate-higher-order-1d},
		\EQn{\label{445}
			\normbb{\int_0^te^{i(t-\rho)\partial_x^2}\mathcal{B}(\langle\nabla \rangle^{-2+s}\eta, \eta u)d\rho}_{L_t^{\infty}L_x^2(I\times \R)} \lsm& T^{\frac 12}R\|\eta\|_{L_x^r}^2.
		}
		By the estimates \eqref{442}-\eqref{445} and \eqref{446}, for any $u\in B_{R}$, there exists a constant $C=C(\|\eta\|_{L_x^r})$, such that
		\begin{align}\label{448}
			\norm{\Phi(u)}_{X(I)}&=\norm{\langle\nabla\rangle^s\Phi(u)}_{L_t^{\infty}L_x^2(I\times \R)}\notag\\
&\leq  R+C\delta R+CT^{\frac 12}RN_0^{s+\frac 12}+CT^{\frac 12}R.
		\end{align}
		First, by \eqref{447}, we take large $N_0=N_0(\delta, \|\eta\|_{L_x^r})$ to obtain small $\delta$, such that
		\EQ{
			C\delta \leq \frac 14.
		}
	Then, we take $T=T(N_0, \|\eta\|_{L_x^r})$ small enough so that
\EQ{
CT^{\frac 12}N_0^{s+\frac 12}+CT^{\frac 12}\leq \frac 12.
}
		Therefore, by the above estimates, we have
		\EQ{
			\norm{\Phi(u)}_{L_t^{\infty}H_x^s(I\times \R)}\leq 2R.
		}
		Hence, we have that $\Phi:B_{R}\rightarrow B_{R}$. Therefore, we complete the proof of local well-posedness by applying contraction mapping principle. Further, we can obtain the global well-posedness by the same way in subsection \ref{GWP}.
	\end{proof}

\subsection{Ill-posedness in $H_x^{s}(\R)$, $s>\frac 52-\frac 1r$}
Finally, we give the proof of the result that for any $s>\frac 52-\frac 1r$, there exists some $\eta \in L_x^r(\R)$ with $1<r\leq 2$, such that the equation \eqref{eq:NLS} is ill-posed in $H_x^{s}(\R)$, which shall finish the proof of Theorem \ref{theorem 1}.

\begin{proof}[\bf{Proof}]
On one hand, we choose the initial data $u_0\in \mathcal{S}$ such that
		$$u_0=P_{\leq 1}u_0, \quad \widehat{u}_0\geq 0, \mbox{ and } \int_{\R} \widehat{u}_0(\xi)\,d\xi>0.$$
		(For example, $u_0(x):=P_{\leq 1}e^{-|x|^2}$).
On the other hand, we choose the spatial potential
\EQ{
\eta(x)=M^{\frac 1r}\mathscr{F}^{-1}\big(\chi_{\frac 12\leq |\cdot|\leq 2}(\xi)\big)(M x),
}
where $M$ is a large constant decided later.
Recall that the function $\chi_{\frac 12\leq |\cdot|\leq 2(\xi)}$ denotes
	$$
	\chi_{\frac 12\leq |\cdot| \leq 2}(\xi) =\left\{ \aligned
	&1, \quad \frac 12\leq |\xi| \leq 2,
	\\
	&0, \quad |\xi|\leq \frac 12-\frac 14 \mbox{ or } |\xi|\geq 2+\frac 14.
	\endaligned
	\right.
	$$
Then we have
\EQ{
\widehat{\eta}(\xi)=M^{-1+\frac 1r}\chi_{\frac 12\leq |\cdot|\leq 2}\big(\frac{\xi}{M}\big).
}
Moreover, noting $\chi_{\frac 12\leq |\cdot|\leq 2}(\xi)$ is a Schwartz function, hence for any $r>1$,
\EQ{
\norm{\eta}_{L_x^r}=\norm{\mathscr{F}^{-1}\big(\chi_{\frac 12\leq |\cdot|\leq 2}(\xi)\big)}_{L_x^r}<\infty.
}
Now, we define
		\EQ{
			B [u_0]&=\int_0^t e^{-is\partial_x^2}(\eta e^{is\partial_x^2}u_0)ds.
		}
Next, we aim to prove that for any $T>0$ and $s>\frac 52-\frac 1r$,
\EQ{
\sup\limits_{t\in [0,T]}\|B[u_0]\|_{H_x^{s}(\R)}\rightarrow \infty, \mbox{ as }M\rightarrow \infty.
}
For our purpose, we set
$$
t\triangleq \frac 1{M^2},
$$
and
$$
\Omega=\{\xi: \sqrt{\frac {\pi}{3}}M\leq|\xi|\leq \sqrt{\frac {\pi}{2}}M\}.
$$
For $B[u_0]$, by the integration-by-parts and the choice of $u_0$ and $\eta$, we have
\begin{align}\label{621656-1}
\widehat{B [u_0]}(\xi)=&M^{-1+\frac 1r}\int_0^t \int_{\xi=\xi_1+\xi_2}e^{is(|\xi|^2-|\xi_2|^2)}\chi_{\frac 12\leq |\cdot|\leq 2}\big(\frac{\xi_1}{M}\big)\widehat{u}_0(\xi_2)d\xi_2ds\notag\\
=&M^{-1+\frac 1r}\int_{\xi=\xi_1+\xi_2}\frac{e^{it(|\xi|^2-|\xi_2|^2)}-1}{i(|\xi|^2-|\xi_2|^2)}\chi_{\frac 12\leq |\cdot|\leq 2}\big(\frac{\xi_1}{M}\big)\widehat{u}_0(\xi_2)d\xi_2.
\end{align}
Hence, taking the real part of $\widehat{B [u_0]}(\xi)$, we have
\begin{align}\label{6121-1}
\mbox{Re}\big(\widehat{B [u_0]}\big)(\xi)=M^{-1+\frac 1r}\int_{\xi=\xi_1+\xi_2}\frac{\sin[t(|\xi|^2-|\xi_2|^2)]}{|\xi|^2-|\xi_2|^2}\chi_{\frac 12\leq |\cdot|\leq 2}\big(\frac{\xi_1}{M}\big)\widehat{u}_0(\xi_2)d\xi_2.
\end{align}
By the mean value theorem, we have
\EQ{
\sin[t(|\xi|^2-|\xi_2|^2)]=\sin (t|\xi|^2)+O(t|\xi_2|^2),
}
where $|O(t|\xi_2|^2)|\leq t|\xi_2|^2 $.
Noting that $t|\xi|^2\in(\frac {\pi}3, \frac {\pi}2)$ for $\xi \in \Omega$, and taking $M$ large enough such that $t|\xi_2|^2\leq M^{-2}\leq \frac 14$, then we can get that
\begin{align}\label{621-1}
\sin[t(|\xi|^2-|\xi_2|^2)]\geq \frac 14.
\end{align}
By the estimates \eqref{6121-1} and \eqref{621-1}, we obtain that for $\xi \in \Omega$,
\begin{align}\label{6122-1}
\mbox{Re}\big(\widehat{B [u_0]}\big)(\xi)\geq \frac {1}{2\pi} M^{-3+\frac 1r}\int_{\xi=\xi_1+\xi_2}\chi_{\frac 12\leq |\cdot|\leq 2}\big(\frac{\xi_1}{M}\big)\widehat{u}_0(\xi_2)d\xi_2>0.
\end{align}
Further, the above estimate yields that
\begin{align*}
\|B [u_0]\|_{H_x^{s}(\R)}\geq& CM^{s-3+\frac 1r}\norm{\int_{\xi=\xi_1+\xi_2}\chi_{\frac 12\leq |\cdot|\leq 2}\big(\frac{\xi_1}{M}\big)\widehat{u}_0(\xi_2)d\xi_2}_{L_{\xi}^2(\Omega)}\\
\geq & CM^{s-3+\frac 1r}M^{\frac 12}\\
=&CM^{s-\frac 52+\frac 1r},
\end{align*}
where $C=C(\int_{\R} \widehat{u}_0(\xi)\,d\xi)>0$ is a finite constant. Hence, for any $T>0$ and $s>\frac 52-\frac 1r$,
\begin{align}\label{Bu0}
\sup\limits_{t\in [0,T]}\|B [u_0]\|_{H_x^{s}(\R)}\rightarrow \infty, \mbox{ as }M\rightarrow \infty.
\end{align}
The proof of ill-posedness is done by applying Lemma \ref{ill-posed}. Hence, we are done proving all the results in theorem \ref{theorem 1}.
\end{proof}

	\section{The proof of Theorem \ref{theorem 3} ($r> 2$)}\label{pf of theorem 3}

\subsection{Global well-posedness in $H_x^{2}(\R)$} We firstly prove that if $\eta \in L_x^r(\R)$ for $r>2$, then the equation \eqref{eq:NLS} is globally well-posed in $H_x^{2}(\R)$.
First of all, we provide some necessary space-time estimates.
For the reader's convenience, let us review the resonant and non-resonant decomposition in  Lemma \ref{resoant-decom} for $s=2$.
	\begin{lem}\label{resoant-decom-2} Let $u(t,x)$ satisfy the following integral equation
\EQn{\label{Duhamel-NLS-2}
		u(t)=e^{it\partial_x^2}u_0+i\int_{0}^t e^{i(t-\rho)\partial_x^2}(\eta u)(\rho)d\rho.
	}
Then we have
		\EQn{\label{Duhamel-NLS-normal-2}
			\langle\nabla \rangle^{2} u(t,x)=&\langle\nabla \rangle^{2}e^{it\partial_x^2}u_0(x)
			-e^{it\partial_x^2}\widetilde{\mathcal{B}}(\eta, u_0(x))+\widetilde{\mathcal{B}}(\eta, u(t,x))\\
			&+i\int_0^te^{i(t-\rho)\partial_x^2}\langle\nabla \rangle^{2}\widetilde{\mathcal{R}}(\eta, u(\rho, x))d\rho\\
			&-i\int_0^te^{i(t-\rho)\partial_x^2}\widetilde{\mathcal{B}}(\eta, \eta u)(\rho, x)d\rho.
		}
Correspondingly, $\widetilde{\mathcal{B}}(f, g)(x)$ and $\widetilde{\mathcal{R}}(\eta, u)$ are defined as follows,
\EQ{
				\widetilde{\mathcal{B}}(f, g)(x) := \int_{\xi=\xi_1+\xi_2} e^{ix(\xi_1+\xi_2)} \tilde{m}(\xi_1, \xi_2)  \widehat{f}(\xi_1)\widehat{g}(\xi_2)d\xi_1d\xi_2;
			}
\EQ{
				\widetilde{\mathcal{R}}(\eta, u):= P_{\leq N_0}(\eta u)+P_{\geq N_0}\sum_{M\gtrsim N}P_N(\eta P_Mu),
			}
where $N_0\in 2^{\N}$ and the multiplier $\tilde{m}(\xi_1, \xi_2)$ is the following
		\EQ{
			\tilde{m}(\xi_1, \xi_2):=\frac{\langle\xi\rangle^2}{|\xi|^2-|\xi_2|^2}\phi_{\geq N_0}(|\xi|)\phi_{\ll 1}\Big(\frac{|\xi_2|}{|\xi|}\Big), {\mbox{ with }}  \xi=\xi_1+\xi_2.
		}
		
	\end{lem}
 Below, we give the estimates for each of the terms in \eqref{Duhamel-NLS-normal-2}.

	\subsubsection{Boundary terms}
	\begin{lem}[Boundary terms]\label{lem:nonlinear-estimate-boundary-1d-2}
		Let $r>2$, and $I \subset\R^+$ be an interval containing $0$. Then, for any $N_0\in2^\N$,
		\EQn{\label{32511-12}
			\normb{e^{it\partial_x^2}\widetilde{\mathcal{B}}(\eta, u_0)}_{L_t^{\infty}L_x^2(I\times \R)} \lsm \|P_{\geq N_0}\eta\|_{L_x^{r}}\|u\|_{L_t^{\infty}H_x^{2}},
		}
		and
		\EQn{\label{32512-12}
			\norm{\widetilde{\mathcal{B}}(\eta, u(t))}_{L_t^{\infty}L_x^2(I\times \R)} \lsm  \|P_{\geq N_0}\eta\|_{L_x^{r}}\|u\|_{L_t^{\infty}H_x^{2}}.
		}
	\end{lem}
	\begin{proof}
		Using Strichartz's estimates, Lemma \ref{lem:Coifman-Meyer} and Sobolev's inequality, we obtain
		\begin{align}\label{32511-112}
			\normb{e^{it\partial_x^2}\widetilde{\mathcal{B}}(\eta, u_0)}_{L_t^{\infty}L_x^2(I\times \R)}
			\lsm &\|P_{\geq N_0} \eta\|_{L_x^{r}}\|u_0\|_{L_x^{\frac{2r}{r-2}}}\notag\\
			\lsm &\|P_{\geq N_0} \eta\|_{L_x^{r}}\|u\|_{L_t^{\infty}H_x^{2}},
		\end{align}
	This gives \eqref{32511-12}. \eqref{32512-12} can be proved by the same way as above. Hence, we complete the proof of this lemma.
	\end{proof}
	
	\subsubsection{Resonance term and low frequency term}
	\begin{lem}\label{lem:nonlinear-estimate-resonance-1d-2}
		Let $r>2$, and $I=[0, T) \subset\R^+$ with $T<1$. Then, for any $N_0\in2^\N$,
		\EQ{
			\normbb{\int_{0}^t e^{i(t-\rho)\partial_x^2}\langle\nabla \rangle^{2}\widetilde{\mathcal{R}}(\eta, u) d\rho}_{L_t^{\infty}L_x^2(I\times \R)} \lsm T^{1-\frac 1{2r}}N_0^{2}\|\eta\|_{L_x^{r}}\|u\|_{L_t^{\infty}H_x^{2}}.
		}
	\end{lem}
	\begin{proof}
		Recalling that
		$
		\widetilde{\mathcal{R}}(\eta, u):= P_{\leq N_0}(\eta u)+P_{\geq N_0}\sum_{M\gtrsim N}P_N(\eta P_Mu).
		$
		Then we have
		\EQnnsub{
			\Big\|\int_{0}^t& e^{i(t-\rho)\partial_x^2}\langle\nabla \rangle^{2}\widetilde{\mathcal{R}}(\eta, u) d\rho\Big\|_{L_t^{\infty}L_x^2}\notag\\
			\lsm & \Big\|\int_{0}^t e^{i(t-\rho)\partial_x^2}\langle\nabla \rangle^{2} P_{\leq N_0}(\eta u)d\rho\Big\|_{L_t^{\infty}L_x^2}\label{63002-2}\\
			&+\Big\|\int_{0}^t \sum_{M\gtrsim N}e^{i(t-\rho)\partial_x^2}\langle\nabla \rangle^{2} P_{\geq N_0}P_N(\eta P_Mu)d\rho\Big\|_{L_t^{\infty}L_x^2}\label{63003-2}.
		}
		For the term \eqref{63002-2}, using the Strichartz and Bernstein estimates,
		\begin{align}\label{63001-2}
			\eqref{63002-2}\lsm&\|\langle\nabla \rangle^{2} P_{\leq N_0}(\eta u)\|_{L_t^{1}L_x^2}\notag\\
			\lsm & N_0^{2}\|\eta u\|_{L_t^{1}L_x^2}\notag\\
			\lsm & TN_0^{2}\|\eta\|_{L_x^r}\|u\|_{L_{t}^{\infty}H_x^{2}}.
		\end{align}
		For the term \eqref{63003-2}, by the duality formula, Strichartz's estimates, and  Lemma \ref{lem:Schur}, we have
		\begin{align}\label{32514-122}
			\eqref{63003-2}
			\lsm& \sup_{\norm{h}_{L_x^2\leq 1}}\sum _{N\lsm M}\normbb{\Big\langle\int_{0}^t\langle\nabla \rangle^{2} e^{i(t-\rho)\partial_x^2}P_N(\eta P_Mu )d\rho, h\Big\rangle}_{L_t^{\infty}}\notag\\
			\lsm &\sup_{\norm{h}_{L_x^2\leq 1}}\sum _{N\lsm M}\frac{\langle N \rangle^{2}}{\langle M \rangle^{2}}\normbb{\int_{0}^t e^{i(t-\rho)\partial_x^2}P_N(\eta \langle M \rangle^{2}P_Mu )d\rho}_{L_t^{\infty}L_x^2}\|P_N h\|_{L_x^2}\notag\\
			\lsm &\sup_{\norm{h}_{L_x^2\leq 1}}\sum _{N\lsm M}\frac{\langle N \rangle^{2}}{\langle M \rangle^{2}} \norm{\eta \langle M \rangle^{2}P_Mu}_{L_t^{\frac{2r}{2r-1}}L_x^{\frac {2r}{2+r}}}\|P_N h\|_{L_x^2}\notag\\
			\lsm &\norm{\eta \langle M \rangle^{2}P_Mu}_{l_M^2L_t^{\frac{2r}{2r-1}}L_x^{\frac {2r}{2+r}}}.
		\end{align}
		Furthermore, by the H\"{o}lder, Minkowski and Sobolev inequalities, and  Lemma \ref{lem:littlewood-Paley}, we get
		\begin{align}\label{32515-12}
			\norm{\eta \langle M \rangle^{2}P_Mu}_{l_M^2L_t^{\frac{2r}{2r-1}}L_x^{\frac {2r}{2+r}}}\lsm& T^{1-\frac 1{2r}}\|\eta\|_{L_x^r}\|\langle M \rangle^{2}P_Mu\|_{L_t^{\infty}L_x^{2}l_M^2}\notag\\
			\lsm & T^{1-\frac 1{2r}}\|\eta\|_{L_x^r}\|u\|_{L_t^{\infty}F_{2}^{2, 2}}\notag\\
			\lsm &T^{1-\frac 1{2r}}\|\eta\|_{L_x^r}\|\langle \nabla \rangle^{2}u\|_{L_t^{\infty}L_x^{2}}.
		\end{align}
		Further, by the above two estimates,
		\begin{align}\label{63006-2}
			\eqref{63003-2}\lsm T^{1-\frac 1{2r}}\|\eta\|_{L_x^r}\|u\|_{L_t^{\infty}H_x^{2}}.
		\end{align}
	Collecting the estimates \eqref{63001-2} and \eqref{63006-2}, we finish the proof of this lemma.
	\end{proof}
	\subsubsection{High-order terms}
	\begin{lem}[High-order terms]\label{lem:nonlinear-estimate-higher-order-1d-2}
		Let $r>2$, and $I=[0, T) \subset\R^+$ with $T<1$. Then
		\EQ{
			\normbb{\int_0^te^{i(t-\rho)\partial_x^2}\widetilde{\mathcal{B}}(\eta, \eta u)d\rho}_{L_t^{\infty}L_x^2(I\times \R)} \lsm& T^{\frac 12}\|\eta\|_{L_x^r}^2\|u\|_{L_t^{\infty}H_x^{2}}.
		}
	\end{lem}
	\begin{proof}
		When $r\geq 4$, by Strichartz's estimates and the Sobolev inequality,
\begin{align}\label{113-22}
\normbb{\int_0^te^{i(t-\rho)\partial_x^2}\widetilde{\mathcal{B}}(\eta, \eta u)d\rho}_{L_t^{\infty}L_x^2}\lsm &\normbb{\widetilde{\mathcal{B}}(\eta, \eta u)}_{L_t^{1}L_x^2}\notag\\
\lsm &\norm{\eta}_{L_x^r}^2\norm{u}_{L_t^{1}L_x^{\frac {2r}{r-4}}}\notag\\
\lsm & T\norm{\eta}_{L_x^r}^2\norm{u}_{L_t^{\infty}H_x^{2}}.
\end{align}
		When $2<r<4$, by the same way as above,
\begin{align}\label{113-222}
\normbb{\int_0^te^{i(t-\rho)\partial_x^2}\widetilde{\mathcal{B}}(\eta, \eta u)d\rho}_{L_t^{\infty}L_x^2}\lsm &\normbb{\widetilde{\mathcal{B}}(\eta, \eta u)}_{L_t^{\frac{4r}{5r-4}}L_x^{\frac r2}}\notag\\
\lsm &T^{\frac{5r-4}{4r}}\norm{\eta}_{L_x^r}^2\norm{u}_{L_{t, x}^{\infty}}\notag\\
\lsm & T^{\frac{5r-4}{4r}}\norm{\eta}_{L_x^r}^2\norm{u}_{L_t^{\infty}H_x^{2}}.
\end{align}
	Combining \eqref{113-22} and \eqref{113-222}, this gives the proof of this lemma.
	\end{proof}

Based on the above several estimates, the local-posedness in $H_x^{2}(\R)$ can be obtained by the standard contraction mapping principle, see for example the proof in subsection \ref{LWP-1}. Further, by the same way in subsection \ref{GWP}, we obtain the global well-posedness. Here, we omit the details.

\subsection{Ill-posedness in $H_x^{s}(\R)$, $s>2$}
Finally, we prove that for any $s>2$, there exists some $\eta \in L_x^r(\R)$ with $r>2$, such that the equation \eqref{eq:NLS} is ill-posed in $H_x^{s}(\R)$.
We prove the result in the similar way as before.
\begin{proof}[\bf{Proof}]
For our purpose, we set the parameters $M, N, L\geq 1$, which shall be determined later. On one hand, we choose the initial data
\EQ{
u_0(x):=\mathscr{F}^{-1}\big(L^{-\frac 12-s }\chi_{L\leq |\cdot|\leq 2L}(\xi)\big)(x).
}
Then we have
\EQ{\|u_0\|_{H_x^{s}}^2=\|\langle\xi\rangle^s \widehat{u_0}(\xi)\|_{L_{\xi}^2}^2\sim 1.
}
On the other hand, we choose the potential
\EQ{
\eta(x)=N^{-1+\frac 1r}\mathscr{F}^{-1}\big(\chi_{\sqrt{\frac {\pi}{3}}M\leq |\cdot|\leq \sqrt{\frac {\pi}{3}}M+N}(\xi)\big)(x).
}
Then we have
\EQ{
\widehat{\eta}(\xi)=N^{-1+\frac 1r}\chi_{\sqrt{\frac {\pi}{3}}M\leq |\cdot|\leq \sqrt{\frac {\pi}{3}}M+N}(\xi).
}
Moreover, noting $\chi_{\sqrt{\frac {\pi}{3}}M\leq |\cdot|\leq \sqrt{\frac {\pi}{3}}M+N}(\xi)$ is a Schwartz function, hence for any $r>2$, we have
\EQ{
\norm{\eta}_{L_x^r}\lsm \norm{\widehat{\eta}}_{L_{\xi}^{r'}}\lsm N^{-1+\frac 1r} N^{\frac 1{r'}} =1,
}
where $r'$ satisfies $\frac 1r+\frac 1{r'}=1$.

Now, we define
\EQ{
C [u_0]&=\int_0^t e^{-is\partial_x^2}(\eta e^{is\partial_x^2}u_0)ds.\\
}
We aim to prove that for any $T>0$ and $s>2$,
\EQ{
\sup\limits_{t\in [0,T]}\|C[u_0]\|_{H_x^{s}(\R)}\rightarrow \infty, \mbox{ as }M\rightarrow \infty.
}
For our purpose, we set
$$
t\triangleq \frac 1{M^2},
$$
and
$$
\Omega=\{\xi: \sqrt{\frac {\pi}{3}}M+\frac N4\leq|\xi|\leq \sqrt{\frac {\pi}{3}}M+\frac 34N\},
$$
For $C[u_0]$, by the integration-by-parts and the choice of $u_0$ and $\eta$, we have
\begin{align}\label{621656}
\widehat{C [u_0]}(\xi)=&N^{-1+\frac 1r}L^{-\frac 12-s }\int_0^t \int_{\xi=\xi_1+\xi_2}e^{is(|\xi|^2-|\xi_2|^2)}\chi_{\sqrt{\frac {\pi}{3}}M\leq |\cdot|\leq \sqrt{\frac {\pi}{3}}M+N}(\xi_1)\chi_{L\leq |\cdot|\leq 2L}(\xi_2)d\xi_2ds\notag\\
=&N^{-1+\frac 1r}L^{-\frac 12-s }\int_{\xi=\xi_1+\xi_2}\frac{e^{it(|\xi|^2-|\xi_2|^2)}-1}{i(|\xi|^2-|\xi_2|^2)}\chi_{\sqrt{\frac {\pi}{3}}M\leq |\cdot|\leq \sqrt{\frac {\pi}{3}}M+N}(\xi_1)\chi_{L\leq |\cdot|\leq 2L}(\xi_2)d\xi_2.
\end{align}
Hence, taking the real part of $\widehat{C [u_0]}(\xi)$, we have
\begin{align}\label{6121}
\mbox{Re}\big(\widehat{C [u_0]}\big)(\xi)=N^{-1+\frac 1r}L^{-\frac 12-s }\int_{\xi=\xi_1+\xi_2}\frac{\sin[t(|\xi|^2-|\xi_2|^2)]}{|\xi|^2-|\xi_2|^2}\chi_{\sqrt{\frac {\pi}{3}}M\leq |\cdot|\leq \sqrt{\frac {\pi}{3}}M+N}(\xi_1)\chi_{L\leq |\cdot|\leq 2L}(\xi_2)d\xi_2.
\end{align}
By the mean value theorem, we have
\EQ{
\sin[t(|\xi|^2-|\xi_2|^2)]=\sin (t|\xi|^2)+O(t|\xi_2|^2).
}
Now, we take $L=\frac N8 \ll M$.  Noting that if $\xi \in \Omega$, then  $t|\xi|^2\sim \frac{\pi}{3}$, which further implies $\sin (t|\xi|^2)\ge \frac 12$. Moreover, by $L\ll M$, we have $t|\xi_2|^2\sim \frac {L^2}{M^2}\ll 1$. Hence, we conclude that
\begin{align}\label{621}
\sin[t(|\xi|^2-|\xi_2|^2)]\geq \frac 14.
\end{align}
By the estimates \eqref{6121} and \eqref{621}, we obtain
\begin{align}\label{6122}
\mbox{Re}\big(\widehat{C [u_0]}\big)(\xi)\geq \frac 14 N^{-1+\frac 1r}L^{-\frac 12-s }\int_{\xi=\xi_1+\xi_2}\frac{1}{|\xi|^2-|\xi_2|^2}\chi_{\sqrt{\frac {\pi}{3}}M\leq |\cdot|\leq \sqrt{\frac {\pi}{3}}M+N}(\xi_1)\chi_{L\leq |\cdot|\leq 2L}(\xi_2)d\xi_2.
\end{align}
Further, noting $\mbox{Re}\big(\widehat{C [u_0]}\big)(\xi)>0$, the above inequality yields that
\begin{align*}
\|C [u_0]\|_{H_x^{s}(\R)}= \norm{\langle\xi\rangle^{s}\widehat{C[u_0]}(\xi)}_{L_{\xi}^2(\R)}\geq \norm{\langle\xi\rangle^{s}\mbox{Re}\big(\widehat{C[u_0]}\big)(\xi)}_{L_{\xi}^2(\R)}.
\end{align*}
Finally, combing the estimate \eqref{6122}, we get
\begin{align*}
\|C [u_0]\|_{H_x^{s}(\R)}\geq& C M^{s}N^{-1+\frac 1r}L^{-\frac 12-s }\norm{\int_{\R}\frac{1}{|\xi|^2-|\xi_2|^2}\chi_{L\leq |\cdot|\leq 2L}(\xi_2)d\xi_2}_{L_{\xi}^2(\Omega)}\\
\geq & CM^{s-2}N^{-1+\frac 1r}L^{-\frac 12-s }\norm{\int_{\R}\chi_{L\leq |\cdot|\leq 2L}(\xi_2)d\xi_2}_{L_{\xi}^2(\Omega)}\\
\geq & CM^{s-2}N^{-1+\frac 1r}L^{-\frac 12-s }LN^{\frac 12}\\
\geq & C(N, L) M^{s-2},
\end{align*}
where $C(N, L)>0$ is a finite constant. Hence, any $T>0$ and $s>2$, we have
\begin{align}\label{Bu0-1D2}
\sup\limits_{t\in [0,T]}\|C [u_0]\|_{H_x^{s}(\R)}\rightarrow \infty, \mbox{ as }M\rightarrow \infty.
\end{align}
The proof of ill-posedness is done by applying Lemma \ref{ill-posed}. Hence, we complete the proof of Theorem \ref{theorem 3}.	
\end{proof}

\appendix
\renewcommand{\theequation}{A.\arabic{equation}} 
\setcounter{equation}{0} 
\renewcommand{\thelemma}{A.\arabic{lemma}}
\setcounter{lemma}{0}
\section*{Appendix}
\addcontentsline{toc}{section}{Appendix}


Here, we aim to prove Lemma \ref{vector-estimate}.
Let $\varPsi$ be a Schwartz function on $\R^d$ whose Fourier transform is supported in the annulus $\frac67 \le |\xi| \le 2$ and is equal to $1$ on $1 \le |\xi| \le \frac{12}{7}$, and  satisfies $\sum_{j\in \Z} \hat{\varPsi}( 2^{-j}\xi)=1$ for $\xi\ne0$.
Let $\varPhi$ satisfy $\hat{\varPhi}(\xi)=\sum_{j\le 0} \hat{\varPsi}( 2^{-j}\xi)$ and $\hat{\varPhi}(0)=1$.
We define the Littlewood-Paley dyadic operators for $j \in \Z$
$$
\varDelta_{j}^{\varPsi}f:=f*\varPsi_{2^{-j}} \text{  and  } S_{j}^{\varPhi}f:=f*\varPhi_{2^{-j}},
$$
where $\varPsi_{2^{-j}}(x)=2^{jd}\varPsi(2^jx)$ and $\varPhi_{2^{-j}}(x)=2^{jd}\varPhi(2^jx)$.

Next, as a result related to the vector-valued Littlewood-Paley theorem and maximal function estimates in \cite{F-S-71, Gra-15, KenigPonceVega-CPAM-93}, we have the following lemma.
\begin{lemma}\label{lem:App3}
	Let $h \in \R^{d}$ and $\tau^{h}g(x)=g(x-h)$. Then, as above, we can define the Littlewood-Paley operators $\varDelta_{j}^{\tau^{h}\varPsi}$, $\varDelta_{j}^{\tau^{h}\varPhi}$ associated with $\tau^{h}\varPsi$, $\tau^{h}\varPhi$ respectively. Moreover, for all $1<p,q<\infty$ and all $L^{p}_{x}L^{q}_{t}$ functions $f$ we have
	\begin{equation}\label{result1}
		\Big\|\Big(\sum_{j\in \Z}|\varDelta_{j}^{\tau^{h}\varPsi}f|^2\Big)^{\frac 12}\Big\|_{L_x^pL^{q}_{t}(\R^d\times \R)}\leq C_{d,p,q}(1+|h|)^{1+d} \|f\|_{L_x^pL^{q}_{t}(\R^d\times \R)},
	\end{equation}
	and
	\begin{equation}\label{result2}
		\Big\|\sup_{j\in \Z}|S_{j}^{\tau^{h}\varPhi}f|\Big\|_{L_x^pL^{q}_{t}(\R^d\times \R)}\leq C_{d,p,q}(1+|h|)^{1+d} \|f\|_{L_x^pL^{q}_{t}(\R^d\times \R)},
	\end{equation}
	where $C_{d,p,q}$ is some constant which only depends on $d,p,q$.
\end{lemma}
\begin{proof}
	Since $\varPsi$ is a Schwartz function, we have
	$$|\varPsi(x-h)|+|\nabla\varPsi(x-h)|\lsm (1+|x-h|)^{-1-d} \lsm (1+|h|)^{1+d}(1+|x|)^{-1-d}.$$
	By the vector-valued version of the Littlewood-Paley theorem (Proposition 6.1.4 in \cite{Gra-15}), the inequality \eqref{result1} holds. Since $\varPhi$ is also a Schwartz function, we have $$|\varPhi(x-h)|\lsm (1+|h|)^{1+d}(1+|x|)^{-1-d},$$and then$$\|\tau^{h}\varPhi\|_{L^1}\lsm (1+|h|)^{1+d}.$$
	Applying Theorem 2.1.10 and Theorem 5.6.6 in \cite{Gra-15}, we have
	\begin{align*}
		\Big\|\sup_{j\in \Z}|S_{j}^{\tau^{h}\varPhi}f|\Big\|_{L_x^pL^{q}_{t}(\R^d\times \R)}=	&\Big\|\sup_{j\in \Z}|f*(\tau^{h}\varPhi)_{2^{-j}}|\Big\|_{L_x^pL^{q}_{t}(\R^d\times \R)}\notag\\
		\le &\Big\|\|\tau^{h}\varPhi\|_{L^1}Mf\Big\|_{L_x^pL^{q}_{t}(\R^d\times \R)}\notag\\
		\lsm &(1+|h|)^{1+d}\|f\|_{L_x^pL^{q}_{t}(\R^d\times \R)},
	\end{align*}
	where $M$ denotes the Hardy-Littlewood maximal operator.
\end{proof}

Now we are in a position to give the proof of Lemma \ref{vector-estimate}. The proof is inspired by \cite{M-P-T-T04}.
\begin{proof}[\bf{Proof of  Lemma \ref{vector-estimate}}.]
For simplicity we treat the $ k=2$ case only.
It will be clear from the proof
that its extension to the $k$-linear case is straightforward.
By the duality, it suffices to proves that for $f_3\in L_x^{p'}L_t^{q'}$,
\begin{align}\label{320-2}
|\Lambda(f_1, f_2, f_3)|\lsm \|f_1\|_{L_x^{p_1}L_t^{q_1}}\|f_2\|_{L_x^{p_2}L_t^{q_2}}\|f_3\|_{L_x^{p'}L_t^{q'}},
\end{align}
where $p'=\frac {p}{p-1}$, $q'=\frac {q}{q-1}$, and
$$
\Lambda(f_1, f_2, f_3)=\int_{\R^{1+d}}T_m(f_1, f_2)(x, t)f_3(x,t)dtdx.
$$

Set $\widetilde{\varDelta_{j}^{\varPsi}}:=\sum\limits_{k=j-5}^{j+5}\varDelta_{k}^{\varPsi}$. By the frequency decomposition, we can rewrite $\Lambda(f_1, f_2, f_3)$ as follows,
\begin{align*}
\Lambda(f_1, f_2, f_3)
=&\sum_{j\in \Z}\Big(\Lambda(S_{j-5}^{\varPhi} f_1, \varDelta_{j}^{\varPsi}f_2, \widetilde{\varDelta_{j}^{\varPsi}}f_3)\\
&\qquad+\Lambda(\widetilde{\varDelta_{j}^{\varPsi}}f_1, \varDelta_{j}^{\varPsi}f_2, S_{j+6}^{\varPhi}f_3)\\
&\qquad+\Lambda(\varDelta_{j}^{\varPsi}f_1, S_{j-5}^{\varPhi}f_2, \widetilde{\varDelta_{j}^{\varPsi}}f_3)\Big).
\end{align*}
To prove \eqref{320-2}, we only treat the first term $\Lambda(S_{j-5}^{\varPhi} f_1, \varDelta_{j}^{\varPsi}f_2, \widetilde{\varDelta_{j}^{\varPsi}}f_3)$, as the remaining two terms can be handled by analogous arguments. We can write
\begin{align}\label{324-2}
\Lambda(S_{j-5}^{\varPhi} f_1, \varDelta_{j}^{\varPsi}f_2, \widetilde{\varDelta_{j}^{\varPsi}}f_3)=&\int_{\R^{1+d}}\int_{\R^{2d}}K(y, y')S_{j-5}^{\varPhi}f_1(t, x-y)\notag\\
 & \qquad\qquad \cdot \varDelta_{j}^{\varPsi}f_2(t, x-y')dydy'\widetilde{\varDelta_{j}^{\varPsi}}f_3(t,x)dtdx\notag\\
=&\int_{\R^{2d}}2^{-jd}2^{(-j+5)d}K(2^{-j+5}y, 2^{-j}y')\int_{\R^{1+d}}S_{j-5}^{\varPhi}f_1(t, x-2^{-j+5}y)\notag\\
& \qquad\qquad \cdot \varDelta_{j}^{\varPsi}f_2(t, x-2^{-j}y')\widetilde{\varDelta_{j}^{\varPsi}}f_3(t,x)dtdxdydy',
\end{align}
where the kernel
\begin{align*}
K(y, y')=\int_{\R^{2d}}m(\xi_1,\xi_2)\hat{\varPhi}(2^{-j+5}\xi_1)\hat{\varPsi}(2^{-j}\xi_2)e^{iy\cdot \xi_1}e^{iy'\cdot \xi_2}d\xi_1d\xi_2.
\end{align*}
From the assumption on $m$, and integration by parts, we get a pointwise bound
\begin{align}\label{323-1}
|2^{-jd}2^{(-j+5)d}K(2^{-j+5}y, 2^{-j}y')|\lsm (1+|y|)^{-2-2d}(1+|y'|)^{-2-2d}.
\end{align}
Furthermore, note that $f*(\tau^{h}g)=(\tau^{h}f)*g=\tau^{h}(f*g)$, we have
\begin{align}\label{324-1}
S_{j-5}^{\varPhi}f_1(x-2^{-j+5}y)= & \tau^{2^{-j+5}y}(f_1*\varPhi_{2^{-j+5}})(x)\notag \\
= &(f_1*\tau^{2^{-j+5}y}\varPhi_{2^{-j+5}})(x)\notag\\
= &(f_1*(\tau^{y}\varPhi)_{2^{-j+5}})(x)\notag\\
= & S_{j-5}^{\tau^{y}\varPhi}f_1(x).
\end{align}
Similarly, we have
\begin{align}\label{325-1}
	\varDelta_{j}^{\varPsi}f_2(x-2^{-j}y')=  \varDelta_{j}^{\tau^{y'}\varPsi}f_2(x).
\end{align}
Hence, by \eqref{324-2}-\eqref{325-1}, using Holder's inequality and Lemma \ref{lem:App3}, we obtain that for any $1<p, p_1, p_2, q, q_1, q_2<\infty$ with $\frac 1 p=\frac 1 {p_1}+\frac 1 {p_2}$ and $\frac 1 q=\frac 1 {q_1}+\frac 1 {q_2}$,
\begin{align*}\label{324-3}
\big|\sum_{j\in \Z}\Lambda(S_{j-5}^{\varPhi} f_1, \varDelta_{j}^{\varPsi}f_2, \widetilde{\varDelta_{j}^{\varPsi}}f_3)\big|\lsm &\int_{\R^{2d}}\int_{\R^{1+d}}(1+|y|)^{-2-2d}(1+|y'|)^{-2-2d}\|S_{j-5}^{\tau^{y}\varPhi}f_1(t,x)\|_{l_j^{\infty}}\notag\\
&\qquad \qquad\cdot\| \varDelta_{j}^{\tau^{y'}\varPsi}f_2(t,x)\|_{l_j^2}\|\widetilde{\varDelta_{j}^{\varPsi}}f_3(t,x)\|_{l_j^2}dtdxdydy'\notag\\
\lsm &\int_{\R^{2d}}(1+|y|)^{-2-2d}(1+|y'|)^{-2-2d}\|S_{j-5}^{\tau^{y}\varPhi}f_1(t,x)\|_{L_x^{p_1}L_t^{q_1}l_j^{\infty}}\notag\\
&\qquad  \cdot\| \varDelta_{j}^{\tau^{y'}\varPsi}f_2(t,x)\|_{L_x^{p_2}L_t^{q_2}l_j^2}\|\widetilde{\varDelta_{j}^{\varPsi}}f_3(t,x)\|_{L_x^{p'}L_t^{q'}l_j^2}dydy'\notag\\
\lsm &\int_{\R^{2d}}(1+|y|)^{-1-d}(1+|y'|)^{-1-d}\|f_1\|_{L_x^{p_1}L_t^{q_1}} \notag\\
&\qquad  \cdot\| f_2\|_{L_x^{p_2}L_t^{q_2}}\|f_3\|_{L_x^{p'}L_t^{q'}}dydy'\notag\\
\lsm & \|f_1\|_{L_x^{p_1}L_t^{q_1}}\| f_2\|_{L_x^{p_2}L_t^{q_2}}\|f_3\|_{L_x^{p'}L_t^{q'}}.
\end{align*}
This proves \eqref{320-2} and thus completes the proof of Lemma \ref{vector-estimate}.
\end{proof}

	\vskip 0.2cm

\end{document}